\documentclass[a4paper]{amsart}
\usepackage[T1]{fontenc}
\usepackage[utf8]{inputenc}
\usepackage[english]{babel}
\usepackage{amsaddr}
\usepackage{amsmath}
\usepackage{amssymb}
\usepackage{amsfonts}
\usepackage{amsrefs}
\usepackage{MnSymbol}
\usepackage{stmaryrd}
\usepackage{mathrsfs} 
\usepackage{natbib}
\usepackage{graphicx,graphics}
\usepackage{slashed}
\usepackage{esint}

\allowdisplaybreaks

\newtheorem{theorem}{Theorem}[section]
\newtheorem{lemma}[theorem]{Lemma}
\newtheorem{prop}[theorem]{Proposition}

\newtheorem{corollary}[theorem]{Corollary}
\newtheorem{definition}[theorem]{Definition}

\newenvironment{theorem*}[1][Theorem]{\begin{trivlist}\itshape
\item[\hskip \labelsep {\bfseries #1}]}{\normalfont\end{trivlist}}
\newenvironment{lemma*}[1][Lemma]{\begin{trivlist}\itshape
\item[\hskip \labelsep {\bfseries #1}]}{\normalfont\end{trivlist}}
\newenvironment{prop*}[1][Proposition]{\begin{trivlist}\itshape
\item[\hskip \labelsep {\bfseries #1}]}{\normalfont\end{trivlist}}
\newenvironment{corollary*}[1][Corollary]{\begin{trivlist}\itshape
\item[\hskip \labelsep {\bfseries #1}]}{\normalfont\end{trivlist}}
\newenvironment{definition*}[1][Definition]{\begin{trivlist}\itshape
\item[\hskip \labelsep {\bfseries #1}]}{\normalfont\end{trivlist}}
\newenvironment{example*}[1][Example]{\begin{trivlist}
\item[\hskip \labelsep {\bfseries #1}]}{\nopagebreak\flushright$\filledmedsquare$\end{trivlist}}
\newenvironment{remark*}[1][Remark]{\begin{trivlist}
\item[\hskip \labelsep {\bfseries #1}]}{\nopagebreak\flushright$\filledmedsquare$\end{trivlist}}

\let\stdparagraph\paragraph
\renewcommand\paragraph{\vspace*{1em}\stdparagraph}

\renewcommand{\cite}[1]{{[\cites{#1}}}

\renewcommand{\phi}{\varphi}

\renewcommand{\rho}{\varrho}

\renewcommand{\theta}{\vartheta}
\newcommand{\eps}{\ensuremath{\varepsilon}}

\renewcommand{\d}{\partial}

\newcommand{\ex}{\exists}
\newcommand{\fa}{\forall}
\newcommand{\lnorm}{\left\lVert}
\newcommand{\rnorm}{\right\lVert}
\newcommand{\lbetr}{\left\lvert}
\newcommand{\rbetr}{\right\lvert}
\newcommand{\norm}[1]{\lnorm {#1}\rnorm}

\newcommand{\abs}[1]{\lbetr {#1}\rbetr}

\renewcommand{\l}{\ensuremath{\left}}
\renewcommand{\r}{\ensuremath{\right}}
\newcommand{\ubr}{\underbrace}

\newcommand{\oli}{\overline}
\newcommand{\opn}{\operatorname}

\newcommand{\then}{\ensuremath{\Rightarrow}}

\newcommand{\sse}{\ensuremath{\subseteq}}

\newcommand{\ssne}{\ensuremath{\subsetneq}}

\newcommand{\vol}{\ensuremath{\mathrm{vol}}}

\newcommand{\Hf}{\ensuremath{\mathfrak{H}}}
\newcommand{\Gf}{\ensuremath{\mathfrak{G}}}
\newcommand{\If}{\ensuremath{\mathfrak{I}}}

\newcommand{\tr}{\ensuremath{\opn{tr}}}
\newcommand{\id}{\opn{id}}
\newcommand{\nn}[1][{}]{\ensuremath{\mathbb{N}^{#1}}}
\newcommand{\rn}[1][{}]{\ensuremath{\mathbb{R}^{#1}}}
\newcommand{\cn}[1][{}]{\ensuremath{\mathbb{C}^{#1}}}

\newcommand{\zn}[1][{}]{\ensuremath{\mathbb{Z}^{#1}}}

\newcommand{\Tn}[1][{}]{\ensuremath{\mathbb{T}^{#1}}}
\newcommand{\Hn}[1][{}]{\ensuremath{\mathbb{H}^{#1}}}
\newcommand{\Tf}[1][{}]{\ensuremath{\mathfrak{T}^{#1}}}

\newcommand{\Tp}[1][{}]{\ensuremath{\mathcal{T}^{#1}}}
\newcommand{\Sp}[1][{}]{\ensuremath{\mathcal{S}^{#1}}}
\newcommand{\Ap}[1][{}]{\ensuremath{\mathcal{A}^{#1}}}
\newcommand{\fock}[1][{}]{\ensuremath{\mathcal{F}_{#1}}}
\newcommand{\expvec}[1][{}]{\ensuremath{\mathcal{E}_{#1}}}
\newcommand{\erf}{\ensuremath{\opn{erf}}}

\newcommand{\res}{\ensuremath{\opn{res}}}
\newcommand{\pr}{\ensuremath{\opn{pr}}}
\newcommand{\ind}{\ensuremath{\opn{ind}}}
\newcommand{\loLc}{\ensuremath{\mathrm{l.o.L.c.}}}
\newcommand{\trKV}{\ensuremath{\tr_{\mathrm{KV}}}}
\newcommand{\Htr}{\ensuremath{\mathrm{Htr}}}

\begin{document}
\title[$\zeta$-reg. and the heat-trace on some compact quantum semigroups]{Zeta-regularization and the heat-trace on some compact quantum semigroups}

\author{Jason Hancox}
\address{Department of Mathematics and Statistics, Lancaster University, LA1 4YF, Lancaster, United Kingdom}
\email{j.hancox@lancaster.ac.uk}

\author{Tobias Hartung}
\address{Department of Mathematics, King's College London, Strand, WC2R 2LS, London, United Kingdom}
\email{tobias.hartung@kcl.ac.uk}

\date{\today}

\begin{abstract}
  Heat-invariants are a class of spectral invariants of Laplace-type operators on compact Riemannian manifolds that contain information about the geometry of the manifold, e.g., the metric and connection. Since Brownian motion solves the heat equation, these invariants can be obtained studying Brownian motion on manifolds. In this article, we consider Brownian motion on the Toeplitz algebra, discrete Heisenberg group algebras, and non-commutative tori to define Laplace-type operators and heat-semigroups on these C*-bialgebras. We show that their traces can be $\zeta$-regularized and compute ``heat-traces'' on these algebras, giving us a notion of dimension and volume. Furthermore, we consider $SU_q(2)$ which does not have a Brownian motion but a class of driftless Gaussians which still recover the dimension of $SU_q(2)$.
\end{abstract}

\maketitle

\tableofcontents

\section{Introduction}\label{sec:intro}
In this article, we want to consider $\zeta$-regularization and the heat-trace in the non-commutative settings of the Toeplitz algebra, discrete Heisenberg group, non-commutative tori, and $SU_q(2)$. $\zeta$-regularization is a means to extend tracial functionals (not necessarily bounded) that are defined on a subalgebra to a larger domain. More precisely, consider an algebra $A$, a subalgebra $A_0$, and a linear functional $\tau:\ A_0\to\cn$ such that $\fa x,y\in A_0:\ \tau(xy)=\tau(yx)$. Given a holomorphic family $\phi:\ \cn\to A$ and $\Omega\sse\cn$ open and connected such that the restriction $\phi|_\Omega$ of $\phi$ to $\Omega$ takes values in $A_0$, we want to consider maximal holomorphic extensions $\zeta(\phi)$ of $\tau\circ\phi|_{\Omega}$. In a way, this is a generalized version of the Riemann $\zeta$-function $\zeta_R$ and its applications like ``$\sum_{n\in\nn}n=\zeta_R(-1)$''. These ideas were pioneered by Ray and Singer~\cite{ray,ray-singer} whose initial works had already been successfully applied by Hawking~\cite{hawking} to compute the energy momentum tensor on the black hole horizon. Since traces are important for studying invariants, $\zeta$-regularization has become an integral part of the pseudo-differential toolkit, especially in geometric analysis. 

As such, the ``classical case'' to consider is where $A=\Psi$ is the algebra of classical [classical is important because we are not looking at the entire algebra of $\psi$dos] pseudo-differential operators on a compact Riemannian $C^\infty$-manifold $M$ without boundary, $A_0$ the dense subalgebra of classical pseudo-differential operators that are trace class on $L_2(M)$, and $\tau$ the canonical trace $\tr$ on the Schatten class $S^1(L_2(M))$. It is then possible to construct this holomorphic family $\phi$ of pseudo-differential operators in such a way that each $\phi(z)$ has affine order $qz+a$ where $q>0$. Then, $\phi(z)$ is of trace class whenever $\Re(z)<\frac{-\dim M-\Re(a)}{q}$ and $\tau\circ\phi$ has a meromorphic extension to $\cn$. Furthermore, all poles are simple and contained in the set $\l\{\frac{j-a-\dim M}{q};\ j\in\nn_0\r\}$. This construction, using the notion of gauged symbols, was introduced by Guillemin~\cite{guillemin}, the residues at the poles give rise to Wodzicki's non-commutative residue~\cite{wodzicki} which (up to a constant factor) is the unique continuous trace on $\Psi$ (if $\dim M>1$), and the constant Laurent coefficients give rise to the Kontsevich-Vishik trace \cite{kontsevich-vishik,kontsevich-vishik-geometry}. It was later shown~\cite{maniccia-schrohe-seiler} that the Kontsevich-Vishik trace (which is unbounded in general) is the unique extension of the canonical trace on $S^1(L_2(M))$ to the subspace of pseudo-differential operators of non-integer order (a dense subspace of $\Psi$ which is not an algebra). The Kontsevich-Vishik trace has also been extended to Fourier integral operators (or, more precisely, ``gauged poly-$\log$-homogeneous distributions'' which contain the gauged Lagrangian distributions studied by Guillemin~\cite{guillemin} which in turn contain Fourier integral operator traces) in Hartung's Ph.D. thesis~\cite{hartung-phd,hartung-scott}.

Families $\phi$ of the form $\phi(z)=TQ^z$, which are constructed using a classical pseudo-differential operator $T$ and complex powers $Q^z$ of an appropriate invertible elliptic operator $Q$~\cite{seeley}, are particularly important example of such $\zeta$-functions. Here, the meromorphic extension of $\tr TQ^z$ is denoted by $\zeta(T,Q)$ and called the $\zeta$-regularized trace of $T$ with weight $Q$. It was shown~\cite{paycha-scott} that the constant term of the Laurent expansion of $\zeta(T,Q)$ centered at zero is of the form $\trKV T-\frac{1}{q}\res(T\ln Q)-\tr(T\pr_{\ker Q})$ where $\trKV$ denotes the Kontsevich-Vishik trace, $\res(T\ln Q)$ is the so called ``trace anomaly'', and $\res$ denotes the extended Wodzicki residue (note that $T\ln Q$ is typically not a pseudo-differential operator and the formula holds only locally as neither $\trKV$ nor $\res$ are globally defined in general). For $T=1$, $\res(\ln Q)$ this is called the logarithmic residue~\cite{okikiolu,scott}. This trace anomaly only appears in the so called ``critical case'' which is if there exists a degree of homogeneity $-\dim M$ in the asymptotic expansion of $T$, i.e., $\zeta(T,Q)$ has a pole in zero.\footnote{Here, we are ignoring the fact that the residue might be zero in which case $\zeta(T,Q)$ is holomorphic in a neighborhood of zero. However, even if this is the case $\zeta(T,Q)(0)$ behaves exactly like you expect the constant Laurent coefficient to behave in the presence of a pole, so for all intents an purposes $\zeta(T,Q)$ has a pole.}

The Kontsevich-Vishik trace of $T$ can be stated in the following form. Let $T$ have symbol $\sigma$ with asymptotic expansion $\sigma(x,\xi)\sim\sum_{j\in\nn_0}\alpha_{m-j}(x,\xi)$ where each $\alpha_{m-j}$ is homogeneous of degree $m-j$ in $\xi$. In other words, 
\begin{align*}
  k(x,y):=&(2\pi)^{-\dim M}\int_{\rn[\dim M]}e^{i\langle x-y,\xi\rangle_{\ell_2(\dim M)}}\sigma(x,\xi)d\xi\\
  \sim&\sum_{j\in\nn_0}\ubr{(2\pi)^{-\dim M}\int_{\rn[\dim M]}e^{i\langle x-y,\xi\rangle_{\ell_2(\dim M)}}\alpha_{m-j}(x,\xi)d\xi}_{=:k_{m-j}(x,y)}
\end{align*}
coincides locally with the kernel of $T$ modulo smoothing operators. Then, there exists $N\in\nn$ (any $N>\dim M+\Re(m)$ will do) such that the operator $T^{\mathrm{reg}}$ with kernel $k^{\mathrm{reg}}:=k-\sum_{j=0}^Nk_{m-j}$ is of trace class and 
\begin{align*}
  \trKV T=\tr T^{\mathrm{reg}}=\int_Mk^{\mathrm{reg}}(x,x)d\vol_M(x).
\end{align*}
This formula has a very important consequence, namely that the Kontsevich-Vishik trace of differential operators ($m\in\nn_0$ and $\fa j\in\nn_{>m}:\ \alpha_{m-j}=0$) vanishes. 

If $Q=\Delta+\pr_{\ker\Delta}$ where $\Delta$ is an elliptic differential operator and $\pr_{\ker\Delta}$ the projection onto its kernel, then $\Gamma(-z)TQ^{-z}$ is the Mellin transform of $Te^{-tQ}$ and, provided $\Delta$ is non-negative, $z\mapsto\Gamma(-z)\zeta(T,Q)(-z)$ is the inverse Mellin transform of $t\mapsto\tr Te^{-tQ}$. For $T=1$ the function $t\mapsto \tr e^{-tQ}$ is called the (generalized) heat-trace generated by $-Q$. An important application of these heat-traces is given in the heat-trace proof~\cite{atiyah-bott-patodi} of the Atiyah-Singer index theorem; namely, if $D$ is a differential operator, then its Fredholm index is given by $\ind D=\tr (e^{-t D^*D}- e^{-t DD^*})$.

Using the inverse Mellin mapping theorem~\cite{azzali-levy-neira-paycha}, it follows that $\tr Te^{-tQ}$ has an asymptotic expansion $\frac{1}{q}\sum_{j\in\nn_0}a_jt^{-d_j}+O(t^{-\gamma})$ where $d_j=\frac{j-a-n}{q}$, $a_j=-\frac{1}{q}\res(TQ^{-d_j})$ for $d_j>0$, and some appropriate $\gamma>\frac{a+\dim M}{q}$. If $T$ is a differential operator plus a trace class operator $T_0$, then we also know that $a_j=\tr T_0-\frac{1}{q}\res(T\ln Q)$ if $d_j=0$.

For instance, let $Q$ be the positive Laplace-Beltrami operator on a compact Riemannian $C^\infty$-manifold $M$ of even dimension and without boundary. Then,
\begin{align*}
  \tr e^{-tQ}=\frac{\vol(M)}{(4\pi t)^{\frac{\dim M}{2}}}+\frac{\mathrm{total\ curvature}(M)}{3(4\pi)^{\frac{\dim M}{2}}t^{\frac{\dim M}{2}-1}}+\mathrm{higher\ order\ terms}
\end{align*}
and, more generally, for $\dim M\in\nn$, the heat-trace has an expansion
\begin{align*}
  \tr e^{-tQ}=(4\pi t)^{-\frac{\dim M}{2}}\sum_{k\in\nn_0}A_kt^{\frac{k}{2}}
\end{align*}
for $t\searrow0$. The $A_k$ are called heat-invariants and are spectral invariants of Laplace type operators $\nabla^*\nabla+V$ generating the corresponding ``heat-semigroup'' where $\nabla$ is a connection on a vector bundle over $M$ and $V$ is a multiplication operator called the potential. More precisely, the heat-invariants are functorial algebraic expressions in the jets of homogeneous components of $Q$, i.e., if $Q$ is geometric, then the heat-invariants carry information about the underlying metric and connection on $M$. We can see this quite nicely in the Laplace-Beltrami case, in which the volume and total curvature appear as lowest order heat-invariants and the dimension of the manifold in the pole order.

These properties of $\zeta$-functions and heat-traces are fundamental in geometric analysis which begs the questions whether or not they extend to non-commutative settings. Such questions have also been studied on the non-commutative torus, the Moyal plane, the Groenewold-Moyal star product, the non-commutative $\phi^4$ theory on the $4$-torus,  and $SU_q(2)$~\cite{azzali-levy-neira-paycha,carey-gayral-rennie-sukochev,carey-rennie-sadaev-sukochev,connes-fathizadeh,connes-moscovici,connes-tretkoff,dabrowski-sitarz-curved,dabrowski-sitarz-asymmetric,fathi,fathi-ghorbanpour-khalkhali,fathi-khalkhali,fathizadeh,fathizadeh-khalkhali-scalar-4,fathizadeh-khalkhali-scalar-2,fathizadeh-khalkhali-gauss-bonnet,gayral-iochum-vassilevich,iochum-masson,levy-neira-paycha,liu,matassa,sadeghi,sitarz,vassilevich-I,vassilevich-II}. 

The non-commutative torus $\Tn[n]_\theta$ is a deformation of the torus $\rn[n]/{\zn[n]}$ using a real anti-symmetric $n\times n$ matrix $\theta$ as a twist. The corresponding C*-algebra $A_\theta$ which has a dense subalgebra consisting of elements $a=\sum_{k\in\zn[n]}a_kU_k$ where $(a_k)_{k\in\zn[n]}$ is in the Schwartz space $\Sp(\zn[n])$, $U_0=1$, each $U_k$ is unitary, and $U_kU_l=e^{-\pi i\langle k,\theta l\rangle_{\ell_2(n)}}U_{k+l}=e^{-2\pi i\langle k,\theta l\rangle_{\ell_2(n)}}U_lU_k$. The construction of $A_\theta$ and its algebra of pseudo-differential operators $\Psi(\Tn[n]_\theta)$ is chosen in such a way that $\theta\to0$ recovers $A_0=C^\infty\l(\rn[n]/{\zn[n]}\r)$ and $\Psi(\Tn[n]_0)$ is the algebra of classical pseudo-differential operators on $\rn[n]/{\zn[n]}$. Then, it is possible to define $\Tn[n]_\theta$ versions of the Wodzicki residue and Kontsevich-Vishik trace and show many of the properties described above. In particular, in~\cite{levy-neira-paycha} it is shown that $\zeta$-functions are meromorphic on $\cn$ with isolated simple poles at $\frac{j-a-n}{q}$ for $j\in\nn_0$ and that the heat-trace pole order is $\frac{n}{2}$.

Similarly, it is possible to introduce a Dirac operator $D_q$ on $SU_q(2)$~\cite{kaad-senior} taking symmetries into account while constructing a twisted modular spectral triple and insuring that the classical limit $q\to 1$ recovers the Dirac operator on $SU(2)$. Using $D_q$, $\zeta$-functions and heat kernels can be constructed on $SU_q(2)$. Heat kernel expansions, heat-traces, $\zeta$-functions and their asymptotics, and their relation to the Dixmier trace (which for pseudo-differential operators coincides with the Wodzicki residue of the $\zeta$-function~\cite{connes-action-functional}) have been studied in this context~\cite{carey-gayral-rennie-sukochev,carey-rennie-sadaev-sukochev,matassa}.

In this article, we want to add another layer of abstraction and consider a number of quantum semigroups. While it is perfectly possible to define a ``twisted'' Laplace-Beltrami operator and, more generally pseudo-differential operators, on the non-commutative torus or $SU_q(2)$ by introducing a non-commutative twist on the classical algebra of pseudo-differential operators on the torus or $SU(2)$, such a construction is not straight forward, if at all possible, for many interesting quantum semigroups. Instead, we want to make use of the fact that Brownian motion solves the heat equation. In other words, the Laplace operator and the heat-semigroup can be recovered using Brownian motion. Hence, our approach in this article is to consider driftless Gaussian processes on quantum semigroups that allow us to define an appropriate notion of Brownian motion and use these Markov semigroups to define Laplace-type operators and ``heat-semigroups''. 

The study of L\'evy processes on *-bialgebras (cf. \cite{schurmann}) gives a very satisfying theory of independent increment processes in the non-commutative framework. This was initiated in the late eighties by Accardi, Sch\"urmann, and von Waldenfels \cite{accardi-schurmann-waldenfels}. The theory generalizes the notion of L\'evy processes on semigroups and allows for various types of familiar L\'evy processes. The most important of these types in this article, and arguably in general, is the notion of a Gaussian L\'evy process. The construction of these L\'evy processes is purely algebraic. Attempts at extending these methods to the C*-algebraic framework have made great progress. Lindsay and Skalski have completed this work relying on the assumption that the generator of the L\'evy process is bounded \cite{lindsay-skalski1,lindsay-skalski2,lindsay-skalski3}. More recently, Cipriani, Franz and Kula \cite{cipriani-franz-kula} have developed a characterization in terms of translation invariant quantum Markov semigroups on compact quantum groups that does not assume the generator to be bounded. At the time of writing an unpublished approach by Das and Lindsay will give a full characterisation for reduced compact quantum groups again which allows unbounded generators. In this article, we will introduce a C*-algebraic L\'evy process methodology that does not rely on the boundedness of the generator but will require the C*-algebra to be universal and ``nicely-generated'' in some sense. This will allow us to develop Gaussian processes on C*-bialgebras (whose generators in general are not bounded) and then by a canonical choice of Gaussian process which we will take to be Brownian motion we will have definitions for a heat-semigroup on our examples of C*-bialgebras.

The Toeplitz algebra is an interesting choice of algebra to consider in this context since it does not have a twist structure of the form allowing us to directly model pseudo-differential operators, yet defining Brownian motion is very natural. Hence, we will start by formally introducing the Toeplitz algebra $\Tp$ and give an overview of convolution semigroups (which contain the notion of L\'evy processes) in section~\ref{sec:semigroups}. Since Brownian motion is classically generated by the Laplace-Beltrami operator, we define a class of operators (polyhomogeneous operators) on the Toeplitz algebra which play a similar role to classical pseudo-differential operators in section~\ref{sec:zeta}, as well as their $\zeta$-functions. Then, we will study the heat-semigroup and $\zeta$-regularized heat-trace in section~\ref{sec:heat}.

While the dynamics of the Toeplitz algebra are generated by the circle $\d B_{\cn}\cong\rn/{2\pi\zn}$, it is not obtained from ``twisting'' the product on $C(\d B_{\cn})$. In particular, it is not merely some $\Tn_\theta$. In order to relate these results to more ``classical'' scenarios, we will consider the discrete Heisenberg group algebra in sections~\ref{sec:heisenberg} and~\ref{sec:heisenberg-Z-complex}, and non-commutative tori in section~\ref{sec:non-com-torus}. In particular, we can relate the heat-traces of discrete Heisenberg group algebras and non-commutative tori to the ``classical'' heat-traces on tori.

Finally, we will consider $SU_q(2)$ which, although being a ``twisted'' manifold, does not have a Brownian motion. Instead all driftless Gaussian semigroups are generated by constant multiples of a unique operator (which is not the Laplacian on $SU(2)$). Still, this family of driftless Gaussians can be regularized and is formally very similar to the Brownian motion on the Toeplitz algebra. 

Our main observations are the following. 

\begin{enumerate}
\item[(i)] On the Toeplitz algebra, $\zeta$-functions of polyhomogeneous operators have at most simple poles in the set $\l\{\frac{-2-d_\iota}{\delta};\ \iota\in I\r\}$ where the $d_\iota$ are the degrees of homogeneity and $\delta$ plays the same role $q$ did above. Furthermore, the ``heat-trace'' can be $\zeta$-regularized, has a first order pole in zero, and the sequence of heat coefficients $(A_k)_{k\in\nn_0}$ satisfies $A_0=-2\pi$ and $\fa k\in\nn:\ A_k=0$. This is exactly what we would expect to see if the Toeplitz algebra were a $2$-dimensional manifold of ``volume'' $-2\pi$ (all other heat coefficients vanishing).
\item[(ii)] In the case of the discrete Heisenberg group algebra $\Hn_N$ we consider two cases; namely, the twist being an abstract unitary or having a complex twist. 
  \begin{enumerate}
  \item[(a)] If the twist is an abstract unitary, then $\zeta$-functions of polyhomogeneous operators have isolated first order poles in the set $\l\{\frac{-2N-1-d_{\iota}}{\delta};\ \iota\in I\r\}$. This corresponds to the classical case of a $2N+1$-dimensional manifold. The heat-trace however is given by $-\tr\circ S$ where $S$ is the heat-semigroup on the $\rn[2N]/{2\pi\zn[2N]}$, i.e., the heat-trace appears as if $\Hn_N$ were a $2N$-torus with all heat-coefficients multiplied by $-1$.
  \item[(b)] If we consider a complex twist, then $\zeta$-functions of polyhomogeneous operators have isolated first order poles in the set $\l\{\frac{-2N-d_{\iota}}{\delta};\ \iota\in I\r\}$ and the heat-trace coincides with the heat-trace on $\rn[2N]/{2\pi\zn[2N]}$. In other words, $\Hn_N$ looks exactly like a $2N$-dimensional torus.
  \end{enumerate}
\item[(iii)] The non-commutative torus $A_\theta^N$ is closely related to the discrete Heisenberg group algebra case. As such we will consider two cases again; (a) ${\Tf}$ twists that are abstract unitaries and (b) ${\Tf}$ complex twists. It is also possible to add another $T'$ complex twists to the case (a) without changing the results.
  \begin{enumerate}
  \item[(a)] If the twists are abstract unitaries, then $\zeta$-functions of polyhomogeneous operators have isolated first order poles in $\l\{\frac{-N-{\Tf}-d_{\iota}}{\delta};\ \iota\in I\r\}$. This corresponds to the classical case of an $N+{\Tf}$-dimensional manifold. The heat-trace however is given by $(-1)^{\Tf}\tr\circ S$ where $S$ is the heat-semigroup on the $\rn[N]/{2\pi\zn[N]}$, i.e., the heat-trace appears as if $A_\theta^N$ were an $N$-torus with all heat-coefficients multiplied by $(-1)^{\Tf}$.
  \item[(b)] If we consider a complex twists, then $\zeta$-functions of polyhomogeneous operators have isolated first order poles in the set $\l\{\frac{-N-d_{\iota}}{\delta};\ \iota\in I\r\}$ and the heat-trace coincides with the heat-trace on $\rn[N]/{2\pi\zn[N]}$. In other words, $A_\theta^N$ looks exactly like an $N$-dimensional torus.
  \end{enumerate}
\item[(iv)] $SU_q(2)$ is a somewhat special case in the list of quantum semigroups we consider here since it does not have a Brownian motion. Hence, there is no heat-semigroup. However, there is still a class of driftless Gaussians that we can consider in lieu of alternatives. Their traces can be $\zeta$-regularized and have a pole in zero which is of order $\frac32$. This corresponds to a $3$-dimensional manifold ($SU_q(2)$ is a twisted $3$-dimensional Calabi-Yau algebra and $SU(2)$ is isomorphic to the $3$-sphere). The sequence of corresponding ``heat-coefficients'' $(A_k)_{k\in\nn_0}$ is given by $A_0=-2\pi^2r^{-\frac{3}{2}}$ and $\fa k\in\nn:\ A_k=0$ where $r\in\rn_{>0}$ is a parameter describing the family of driftless Gaussians on $SU_q(2)$. Hence, we still obtain consistent results regarding dimensionality of $SU_q(2)$ but interpreting the ``heat-coefficient'' $A_0$ as volume would be a bit of a stretch as it is also negative in the $SU(2)$ case. Furthermore, $\zeta$-functions of polyhomogeneous operators have isolated first order poles in the set $\l\{\frac{-3-d_{\iota}}{\delta};\ \iota\in I\r\}$.
\end{enumerate}

\paragraph{\textbf{Acknowledgements}} The authors would like to express his gratitude to Prof. Martin Lindsay and Prof. Simon Scott for inspiring comments and conversations which helped us to develop the work presented in this article. The first author was funded by the Faculty of Science and Technology at Lancaster University.

\section{Convolution semigroups and the Toeplitz algebra}\label{sec:semigroups}
In this section we will introduce the Toeplitz algebra as a C*-bialgebra. This has been introduced previously in~\cite{aukhadiev-grigoryan-lipacheva}. We will characterize the Sch\"urmann triples on this C*-bialgebra which generalizes the notion of L\'evy processes on a compact topological semigroup.

This will lead to a natural choice for Brownian motion and in later sections we will calculate important quantities associated to this semigroup that in the classical setting give information about the structure of the manifold involved.

\begin{definition}
The universal C*-algebra generated by the right shift operator $R:\ \ell_2(\nn_0)\to \ell_2(\nn_0)$ such that 
\begin{align*}
R(\lambda_0,\lambda_1,\dots)= (0,\lambda_0,\lambda_1,\dots)
\end{align*}
is called the Toeplitz algebra and denoted $\Tp$.
\end{definition}

The Toeplitz algebra has a dense *-subalgebra with basis given by $R_{n,m}=R^nR^{*m}$. We will denote this sub *-algebra $\Tp_0$. For a more detailed account of the Toeplitz algebra see \cite{murphy}.

We will proceed to define C*-bialgebras, these are the non-commutative analogue to topological semigroups with identity in the same sense that C*-algebras are a non-commutative analogue to locally compact Hausdorff topological spaces and compact quantum groups are non-commutative analogues to compact groups.

\begin{definition}
A  *-bialgebra is a unital  *-algebra $A$ with unital  *-homomorphisms $\Delta:\ A\to A\otimes A$ and $\eps:\ A\to \cn$ that satisfy
\begin{align*}
(\Delta\otimes \id)\circ\Delta=(\id\otimes \Delta)\circ \Delta \quad\text{ and }\quad (\eps\otimes \id)\circ\Delta=\id=(\id\otimes \eps)\circ \Delta
\end{align*}
where $\otimes$ is the algebraic tensor product.
\end{definition}
\begin{definition}
A C*-bialgebra is a unital C*-algebra $A$ with unital C*-homo--morphisms $\Delta:\ A\to A\otimes A$ and $\eps:\ A\to \cn$ that satisfy 
\begin{align*}
(\Delta\otimes \id)\circ\Delta=(\id\otimes \Delta)\circ \Delta \quad\text{ and }\quad (\eps\otimes \id)\circ\Delta=\id=(\id\otimes \eps)\circ \Delta
\end{align*}
where $\otimes$ is the spatial tensor product.
\end{definition}

If we also required that the sets $\Delta(A)(1\otimes A)$ and $\Delta(A)(1\otimes A)$ were dense in $A\otimes A$ in the definition of C*-bialgebra we would have the definition of a compact quantum group. These conditions are the quantum cancellation properties but will not be required for this.

The map $\Delta$ is called to co-multiplication and the first identity involving only $\Delta$ is called co-associativity. This is to mirror the multiplication of a semigroup. The map $\eps$ is called the co-unit and the second identity is called the co-unital property. This is analogous to the identity element of a semigroup. 

\begin{prop}
The Toeplitz algebra can be given the structure of a C*-bialgebra with co-multiplication $\Delta(R_{n,m})=R_{n,m}\otimes R_{n,m}$ and co-unit $\eps(R_{n,m})=1$ for all $n,m\in\nn_0$. Furthermore, the restriction of these maps to $\Tp_0$ makes $\Tp_0$ a *-bialgebra.
\end{prop}
\begin{proof}
As the Toeplitz algebra is a universal C*-algebra generated by the isometry $R$, we only need to show that $\Delta(R)^*\Delta(R)=I_{\Tp\otimes \Tp}$ and $\eps(R)^*\eps(R)=1$. This is straightforward:
\begin{align*}
\Delta(R)^*\Delta(R)&=(R^*\otimes R^*)(R\otimes R)=R^*R\otimes R^*R=I_{\Tp}\otimes I_{\Tp}=I_{\Tp\otimes\Tp}
\end{align*}
and
\begin{align*}
\eps(R)^*\eps(R)&=(1^*)(1)=1.
\end{align*}
The fact that the maps restricted to $\Tp_0$ gives it the structure of a *-bialgebra is easily seen by the identity $\Delta(R_{n,m})=R_{n,m}\otimes R_{n,m}$.
\end{proof}

Now that we have a *-bialgebra we can appeal to the theory of L\'evy processes on *-bialgebras~\cite{schurmann,lindsay-skalski}.

For a pre-Hilbert space $D$, let $L^*(D)$ denote the set of adjointable operators, that is, linear maps $T:\ D\to D$ such that there exists $T^*:\ D\to D$ such that $\fa x,y\in D:\ \langle x,Ty  \rangle=\langle T^*x,y  \rangle$. This is clearly a unital *-algebra.

\begin{definition}
Let $A$ be a *-bialgebra . A Sch\" urmann triple $(\rho,\eta,L)$ consists of a unital *-homomorphism $\rho:\ A\to L^*(D)$ for some pre-Hilbert space $D$, a $\rho-\eps$ cocycle $\eta:\ A\to D$, i.e.,
\begin{align*}
\eta(ab)=\eta(a)\eps(b)+\rho(a)\eta(b)
\end{align*}
and a *-linear functional $L:\ A\to\cn$ such that
\begin{align*}
L(ab)=L(a)\eps(b)+\eps(a)L(b)+\l\langle \eta(a^*),\eta(b)\r\rangle.
\end{align*}
A Sch\"urmann triple will be called surjective if the cocycle $\eta$ has dense image.
\end{definition}

If we let $\overline{D}$ be the Hilbert space completion of $D$ and we consider unital *-homomorphisms $\rho:\ \Tp_0\to L^*(D)$, we can see that $\rho(R)\in L^*(D)$ is an isometry and can therefore be extended to $B\l(\overline{D}\r)$. As $\Tp_0$ is generated by $R$ we can now use induction on word length to see that $\rho(R_{n,m})$ can be extended to $B\l(\overline{D}\r)$ for all $n,m\in \nn_0$.

Therefore we can replace the pre-Hilbert space in the Sch\"urmann triple definition by a Hilbert space and the adjointable operators by the bounded operators.

We will now proceed to characterize the Sch\"urmann triples on $\Tp_0$. 

\begin{theorem}\label{thm:isolp}
Given an isometry $V\in B(H)$ on some Hilbert space $H$, $h\in H$, and $\lambda\in \rn$, there exists a unique Sch\"urmann triple $(\rho,\eta,L)$ on $\Tp$ such that
\begin{align*}
\rho(R)=V,\quad \eta(R)=h,\quad\text{ and }\quad L(R-R^*)=i\lambda.
\end{align*}
Furthermore, every Sch\"urmann triple arises this way.
\end{theorem}

\begin{proof}
Clearly given any Sch\"urmann triple we can see that $\rho(R)$ is an isometry on some Hilbert space $H$. By definition $\eta(R)$ is an element of $H$ and by *-linearity $L(R-R^*)$ is a purely imaginary number.

Starting with $V\in B(H), h\in H$, and $\lambda \in \rn$, we easily construct $\rho:\Tp\to B(H)$ by universality where $\rho(R)=V$. If we let $\eta(R)=h$, $\eta(R^*)=-Vh$ and $\eta(ab)=\eta(a)\eps(b)+\rho(a)\eta(b)$ for all $a,b\in \Tp_0$, we will see that $\eta(R^*R)=0$ and $\eta:\ \Tp_0\to H$ is well-defined.

Finally, let $L(R-R^*)=i\lambda$, $L(R+R^*)=-\l\langle h,h\r\rangle$, and $L(ab)=L(a)\eps(b)+\eps(a)L(b)+\l\langle \eta(a^*),\eta(b)\r\rangle$ for all $a,b\in \Tp_0$. Again, we see that $L(R^*R)=0$ and $L:\Tp_0\to \cn$ is well-defined.
\end{proof}

\begin{definition}
A convolution semigroup of states is a family of linear functionals $\phi_t:\ A\to \cn$ such that $\phi_t(a^*a)\geq 0$ and $\phi_t(1)=1$ for all $t\geq 0$, i.e., $\phi_t$ is a *-algebra state for all $t\in\rn_{\ge 0}$ and 
\begin{align*}
  \phi_t*\phi_s:=(\phi_s\otimes \phi_t)\circ\Delta=\phi_{t+s},\quad \phi_0=\eps,\quad \text{ and }\quad \lim_{r\to 0}\phi_r(a)=\eps(a)
\end{align*}
for all $t,s\in \rn_{\ge0}$ and $a\in A$.
\end{definition}

\begin{definition}
A generating functional is a linear functional $L:A\to \cn$ such that
$$L(1)=0, \quad L(a^*)=\overline{L(a)},\quad \text{ and } L((a-\eps(a))^*(a-\eps(a)))\geq 0$$
for all $a\in A$.
\end{definition}

Sch\"urmann proved that following are in one-to-one correspondence
\begin{itemize}
\item Sch\"urmann triples on $A$;
\item Convolution semigroups of states on $A$;
\item Generating functionals on $A$.
\end{itemize}
In the classical setting given a L\'evy process $X_t$ on a compact semigroup the associated probabiltity distributions $\mu_t$ form a convolution semigroup of probability measures. This motivates the definition of L\'evy processes on *-bialgebras as states act as a noncommutative analogue to probability measures by results such as the Markov-Riesz-Kakutani theorem.

In the definition of the Sch\"urmann triple the functional $L$ is the generating functional. These will assist us in constructing contraction semigroups of operators. To extend these results to the C*-algebraic level we will introduce the symmetric Fock space.

\begin{definition}
Let $H$ be a Hilbert space. The symmetric Fock space is given by
\begin{align*}
  \cn\Omega\oplus\bigoplus_{n\geq 1} H^{\vee n}
\end{align*}
where $\Omega$ is called the vacuum vector and $H^{\vee n}\sse H^{\otimes n}$ such that elements are unchanged by the action of permutation of tensor factors. The symmetric Fock space of $H$ is denoted by $\Gamma(H)$.

If $H=L_2(\rn_{\ge0}; K)$ for some Hilbert space $K$, we will call $\Gamma(H)=\fock$ and for $I\subseteq \rn_{\ge0}$ call $\Gamma(L_2(I; K))=\fock[I]$.
\end{definition}

Note that the so called exponential property of Fock spaces with
\begin{align*}
  L_2([0,b_1);K)\oplus L_2([b_1,b_2);K)\oplus \dots \oplus L_2([b_n,\infty);K)\cong L_2(\rn_{\ge0};K)
\end{align*}
gives the decomposition
\begin{align*}
  \fock[[0,b_1)]\otimes\fock[[b_1,b_2)]\otimes\dots\otimes \fock[[b_n,\infty)]\cong \fock
\end{align*}
for all $n\in \nn$ and $0< b_1<b_2<\dots<b_n$.

A very important subspace of the Fock space is the the space of exponential vectors given by the linear span of the vectors 
\begin{align*}
  e(u)=\l(1,u,\frac{u^{\otimes 2}}{\sqrt{2}},\dots ,\frac{u^{\otimes n}}{\sqrt{n!}},\dots\r)\in \fock
\end{align*}
for all $u\in L_2(\rn_{\ge0};K)$. This is a dense subspace of $\fock$ and we will denote it by $\expvec$.

Using this characterization of Sch\"urmann processes we can now appeal to the Representation Theorem (Theorem~1.15~\cite{franz}) to realize our L\'evy process on the Fock space. This gives us a family of adapted unital weak*-homomorphisms $j_{s,t}:\ \Tp_0\to L^\dagger(\expvec)$ where $L^\dagger(\expvec)$, in the context of $\expvec$ being a subspace of a Hilbert space, is the family of linear operators $T:\ \expvec\to \fock$ such that the adjoint $T^*$ has domain which contains $\expvec$. 

More specifically, $j_{s,t}:A\ \to L^\dagger(\expvec)$ is a family such that $j_{s,t}(a)$ acts non-identically only on $\fock[[s,t)]$ (adapted), $j_{s,t}(1)=\id_{\fock}$ (unital), satisfies the weak multiplicative property 
  \begin{align*}
    \langle x,j_{s,t}(a^*b)y\rangle=\langle j_{s,t}(a)x,j_{s,t}(b)y\rangle,
  \end{align*}
and is a Fock space L\'evy process, i.e., a family $(j_{s,t})_{0\leq s\leq t}$ of maps $A\to L^\dagger(\expvec)$ such that
\begin{itemize}
\item[(i)] $ j_{t,t} =\id_{\fock}$,
\item[(ii)] $ (j_{r,s}\otimes j_{s,t})\circ \Delta =j_{r,t}$,
\item[(iii)] $\lim_{t\to s}\langle e(0), j_{s,t}(a)e(0))\rangle =1$   and  
\item[(iv)] $\langle e(0), j_{s,t}(a)e(0))\rangle =\langle e(0), j_{s+r,t+r}(a)e(0))\rangle$.
\end{itemize}
The following result of Belton and Wills \cite{belton-wills} tells us that, because the Toeplitz algebra is nicely generated, algebraic unital weak*-homomorphisms are enough to extend to C*-algebraic unital homomorphisms.

\begin{prop}
  There is a one-to-one correspondence between unital weak*-homomorphisms $j:\ \Tp_0\to L^\dagger(\expvec)$ and unital C*-homomorphisms $\hat{j}:\ \Tp\to B(\fock)$.
\end{prop}

\begin{proof}
Clearly given $\hat{j}:\ \Tp\to B(\fock)$ then $j:=\hat{j}|_{\Tp_0}$ defines a unital *-homomorphism $j:\ \Tp_0\to B(\fock)\subseteq L^\dagger(\expvec)$.

Now let $j:\ \Tp_0\to L^\dagger(\expvec)$ this implies that 
\begin{align*}
\norm x^2=\langle x,j(R^*R)x\rangle=\langle j(R)x,j(R)x\rangle=\norm{j(R)x}^2
\end{align*}
for all $x\in \expvec$. Since $\expvec$ is dense in $\fock$, $j(R)$ can be extended to an isometry in $B(\fock)$. Adjointability implies that $j(R^*)$ is also bounded. We can now use induction on word length and linearity to show that $j(R_{n,m})\in B(\fock)$ for all $n,m\in \nn_0$. Now using universality there exists a unital C*-homomorphism $\hat{j}:\ \Tp\to B(\fock)$ such that $\hat{j}(R)=j(R)$.
\end{proof}

\begin{corollary}
There is a one-to-one correspondence between Sch\"urmann triples on $\Tp_0$ and L\'evy processes $j_{s,t}:\ \Tp\to B(\fock)$.
\end{corollary}

Given such a L\'evy process we get a $C_0$-convolution semigroup of states \cite{lindsay-skalski1} on the C*-algebra $\Tp$ this is given by $\phi_t(a)=\langle \Omega,j_{0,t}(a)\Omega\rangle$. Furthermore we get an associated $C_0$-semigroup $T(t):\ \Tp\to \Tp$ given by $T(t):=(\id\otimes \phi_t)\circ \Delta$.

\begin{example*}
Let $H=\cn$, $V=\id_{\cn}$, $h=1$, and $\lambda=0$ from Theorem \ref{thm:isolp}. This is the natural choice of Brownian motion on the Toeplitz algebra. Firstly Sch\"urmann triples are said to be Gaussian if and only if the associated unital *-homomomorphism is of the form $\rho=\epsilon$.

Furthermore, if the $V\in B(H)$ in Theorem \ref{thm:isolp} is chosen to be unitary, then the associated L\'evy process can be restricted to the quotient C*-algebra $\Tp/K(\ell_2(\nn_0))\cong C(\d B_{\cn})$ the continuous functions on the circle group.

In the case above the associated L\'evy process corresponds to the standard Brownian motion on the real line ``wrapped'' around the circle.

More explicitly the Sch\"urmann triple on the dense *-bialgebra $\Tp_0$ is given by
\begin{align*}
\rho(R_{n,m})=\eps(R_{n,m})=1,\quad \eta(R_{n,m})=n-m,\quad \text{ and }\quad L(R_{n,m})=-\frac{(n-m)^2}{2}.
\end{align*}

This has an associated $C_0$-convolution semigroup of states that acts on the dense *-bialgebra by
\begin{align*}
\phi_t(R_{n,m})=e^{-\frac{(n-m)^2}{2}t}.
\end{align*}
\end{example*}

\section{$\zeta$-regularized traces of polyhomogeneous operators}\label{sec:zeta}

In this section, we want to consider a class of operators that generate convolution semigroups on the Toeplitz algebra $\Tp$ which resemble pseudo-differential operators. In particular, the generator of Brownian motion, i.e., our version of the Laplacian, is an operator of this type. We will then show, that these operators have $\zeta$-regularized traces, i.e., analogues of the Kontsevich-Vishik trace and residue trace. 
Recall the following properties of the Toeplitz algebra $\Tp$ and generators of convolution semigroups on $\Tp$.
\begin{enumerate}
\item[(i)] The space $\Tp_0:=\opn{lin}\l\{R_{n,m};\ n,m\in\nn_0\r\}$ is a dense $^*$-subalgebra of $\Tp$.
\item[(ii)] The co-unit $\eps$ satisfies $\fa m,n\in\nn_0:\ \eps(R_{n,m})=1$.
\item[(iii)] The co-multiplication $\Delta$ satisfies 
  \begin{align*}
    \fa m,n\in\nn_0\ \fa\alpha\in\cn:\ \Delta(\alpha R_{n.m})=\alpha R_{n,m}\otimes R_{n,m}.
  \end{align*}
\item[(iv)] Let $L$ be the generating functional of a convolution semigroup $\omega$. Then, the corresponding operator semigroup is given by $t\mapsto (\id\otimes\omega_t)\circ\Delta$ and has generator $(\id\otimes L)\circ\Delta$.
\end{enumerate}

\begin{definition}
  An operator $H$ on $\Tp$ is called polyhomogeneous if and only if there exists a functional $L:\ \Tp_0\to\cn$ such that $H|_{\Tp_0}=(\id\otimes L)\circ\Delta$ and 
  \begin{align*}
    \ex r\in\rn\ \ex I\sse\nn\ \ex\alpha\in \ell_1(I)\ \ex d\in(\cn[2]_{\Re(\cdot)<r})^I\ \fa m,n\in\nn_0:\ L(R_{n,m})=\sum_{\iota\in I}\alpha_\iota \sigma_{d_\iota}(m,n)
  \end{align*}
  where $\cn_{\Re(\cdot)<r}:=\{z\in\cn;\ \Re(z)<r\}$, $\cn[2]_{\Re(\cdot)<r}:=\cn_{\Re(\cdot)<r}\times\cn_{\Re(\cdot)<r}$, $\sigma_{d_\iota}:\ \rn[2]\to\cn$ is homogeneous of degree $d_\iota$, i.e.,
  \begin{align*}
    \fa \lambda\in\rn_{>0}\ \fa\xi\in\rn[2]:\ \sigma_{d_\iota}(\lambda\xi)=\lambda^{d_\iota}\sigma_{d_\iota}(\xi),
  \end{align*}
  and  $\sum_{\iota\in I}\alpha_\iota \sigma_{d_\iota}(m,n)$ is absolutely convergent.
\end{definition}

\begin{example*}
  The generating functional of Brownian motion is given by
  \begin{align*}
    \fa m,n\in\nn_0:\ L_{BM}(R_{n,m})=-\frac{(n-m)^2}{2}=\l(\frac{-1}{2}n^2m^0\r)+\l(n^1m^1\r)+\l(\frac{-1}{2}n^0m^2\r).
  \end{align*}
  Thus, the generator $H_{BM}$ of the Brownian motion semigroup $B$ is polyhomogeneous with finite $I$ and each $d_\iota=2$. 
\end{example*}

Since $2H_{BM}$ is the Laplace-Beltrami on compact Riemannian manifolds without boundary, we obtain the following definition of the Laplacian on $\Tp$.
\begin{definition}
  Let $H_{BM}=(\id\otimes L_{BM})\circ\Delta$ be the generator of Brownian motion. Then, we call the (polyhomogeneous) operator $\Delta_{\Tp}:=2H_{BM}$ the Laplacian on $\Tp$.
\end{definition}

We are now interested in $\zeta$-regularized traces of polyhomogeneous operators. Thus, in order to compute traces, the following results shine a light on their spectral properties.

\begin{lemma}
  Let $L$ be the generating functional of a convolution semigroup of states on $\Tp$ and $H:=(\id\otimes L)\circ\Delta$. Then, $H$ is a closed, densely defined operator and
  \begin{align*}
    \fa n\in\nn\ \fa \lambda\in\rn_{>\omega}:\ \norm{(\lambda-H)^{-1}}\le\frac{1}{\lambda}.
  \end{align*}
\end{lemma}

\begin{proof}
  Let $\phi$ be the convolution semigroup of states generated by $L$. Then, $t\mapsto(\id\otimes\phi_t)\circ\Delta$ satisfies
  \begin{align*}
    \fa t\in\rn_{>0}:\ \norm{T(t)}=\norm{(\id\otimes\phi_t)\circ\Delta}\le\norm{\id\otimes\phi_t}\norm\Delta=1
  \end{align*}
    and is a contraction semigroup on $\Tp$. Since $\Tp$ is a Banach space, the Theorem of Hille-Yosida-Phillips yields the result.
\end{proof}

\begin{lemma}
  Let $H=(\id\otimes L)\circ\Delta$ with $L:\ \Tp_0\to\cn$ linear. Then, the point spectrum $\sigma_p(H)$ of $H$ is given by
  \begin{align*}
    \l\{L(R_{n,m});\ m,n\in\nn_0\r\}\sse\sigma_p(H)
  \end{align*}
  including multiplicities and the spectrum $\sigma(H)$ of $H$ is given by
  \begin{align*}
    \sigma(H)=\oli{\l\{L(R_{n,m});\ m,n\in\nn_0\r\}}.
  \end{align*}
  In particular, if $\l\{L(R_{n,m});\ m,n\in\nn_0\r\}$ is closed in $\cn$, then
  \begin{align*}
    \sigma(H)=\sigma_p(H)=\l\{L(R_{n,m});\ m,n\in\nn_0\r\}
  \end{align*}
  including multiplicities.

  Furthermore, if $\sigma(H)\ssne\cn$, then $H$ is closable.
\end{lemma}

\begin{proof}
  Let $S:=\l\{L(R_{n,m});\ m,n\in\nn_0\r\}$. Then, we observe
  \begin{align*}
    \fa m,n\in\nn_0:\ HR_{n,m}=L(R_{n,m})R_{n,m},
  \end{align*}
  i.e., $S\sse\sigma_p$ and $\oli{S}\sse\sigma(H)$ since the spectrum is always closed.

  Let $\lambda\in\cn\setminus\oli{S}$. Then, $\lambda-H$ is boundedly invertible\footnote{By ``boundedly invertible on $\Tp_0$'' we mean that $\lambda-H:\ \Tp_0\to\Tp_0$ is bijective, i.e., the inverse relation $(\lambda-H)^{-1}:=\{(x,y)\in\Tp[2]_0;\ (y,x)\in\lambda-H\}$ is an operator, and $(\lambda-H)^{-1}:\ \Tp_0\to\Tp_0$ is bounded with respect to the topology induced by $\Tp$.}
    on $\Tp_0$ and, since $\Tp_0$ is dense in $\Tp$, we obtain $\sigma(H)\sse\oli{S}$.

  Finally, assume $\sigma(H)\ssne\cn$ and let $\lambda\in\rho(H)$. Then, $\lambda-H$ is boundedly invertible, i.e., closable. Since $H$ is closable if and only if $\lambda-H$ is closable, we obtain the assertion.
\end{proof}

\begin{example*}
  The Laplacian $\Delta_{\Tp}$ and the heat-semigroup $T$ (generated by $\Delta_{\Tp}$) have pure point spectrum
  \begin{align*}
    \sigma(\Delta_{\Tp})=&\sigma_p(\Delta_{\Tp})=\l\{-(n-m)^2;\ m,n\in\nn_0\r\}\\
    \fa t\in\rn_{\ge0}:\ \sigma(T(t))=&\sigma_p(T(t))=\l\{\exp\l(-(n-m)^2t\r);\ m,n\in\nn_0\r\}
  \end{align*}
  including multiplicities.
\end{example*}

\begin{remark*}
  Here we can see two very important differences to the classical theory. Namely, $\Delta_{\Tp}$ does not have compact resolvent and the heat-semigroup is not a semigroup of trace class operators (in fact, they are not even compact). In particular, this means that the heat-trace will need to be regularized. 
\end{remark*}

Hence, we know which polyhomogeneous operators have pure point spectrum. However, since ``$\tr H=\sum_{\lambda\in\sigma(H)\setminus\{0\}}\mu_\lambda \lambda$'' (where $\mu_\lambda$ denotes the multiplicity of $\lambda$) will not converge in general, the idea is to use a spectral $\zeta$-regularization similar to ``$\sum_{n\in\nn}n=\zeta_R(-1)$'' where $\zeta_R$ is the Riemann $\zeta$-function.

\begin{definition}
  Let $\Gf$ be a holomorphic family of operators satisfying 
  \begin{align*}
    \fa z\in\cn:\ \Gf(z)=(\id\otimes L_{\Gf(z)})\circ\Delta=\Gf_0(z)+\Gf_p(z)
  \end{align*}
 such that each $\l\{L_{\Gf(z)}(R_{n,m});\ m,n\in\nn\r\}$ is closed, $\Gf_0(z)$ is of trace class for all $z\in\cn_{\Re(\cdot)<R}$ with $R\in\rn_{>0}$, and $\Gf_p$ is polyhomogeneous with 
  \begin{align*}
    \fa m,n\in\nn_0:\ L_{p,\Gf(z)}(R_{n,m})=\sum_{\iota\in I}\alpha_\iota(z)\sigma_{d_\iota+\delta z}(m,n)
  \end{align*}
  where each $\alpha_\iota$ is holomorphic and $\delta\in\rn_{>0}$. Then, we call $\Gf$ a gauged polyhomogeneous operator with index set $I$.

  Furthermore, we call $\Gf$ normally gauged if and only if $\delta=1$. $\delta$ is called the gauge scaling.
\end{definition}

\begin{theorem}\label{thm:existence-zeta-function}
  Let $\Gf=\Gf_0+\Gf_p$ be a gauged polyhomogeneous operator on $\Tp$ with
  \begin{align*}
    \fa m,n\in\nn_0:\ L_{p,\Gf(z)}(R_{n,m})=\sum_{\iota\in I}\alpha_\iota(z)\sigma_{d_\iota+\delta z}(m,n)
  \end{align*}
  such that $\sigma(\Gf_p(z))=\sigma_p(\Gf(z))=\{L_{p,\Gf(z)}(R_{n,m});\ m,n\in\nn_0\}$. Then, $\Gf$ is of trace class if $\Re(d_\iota+\delta z)<-2$.

  Furthermore, the meromorphic extension $\zeta(\Gf)$ of
  \begin{align*}
    \cn_{\Re(\cdot)<\frac{-2-\sup\{\Re(d_\iota);\ \iota\in I\}}{\delta}}\ni z\mapsto\tr\Gf(z)\in\cn
  \end{align*}
  exists on a half space $\cn_{\Re(\cdot)<R}$ with $R\in\rn_{>0}$ and has at most simple poles at points in $\l\{\frac{-2-d_\iota}{\delta};\ \iota\in I\r\}$.

\end{theorem}

\begin{proof}
  Since $\Gf_0$ is of trace class on a half space $\cn_{\Re(\cdot)<R}$ with $R\in\rn_{>0}$, it suffices to consider $\Gf_p$. Since $\Gf_p(z)$ has the same spectrum as $D(z):=\sum_{\iota\in I}\alpha_\iota(z)\sigma_{d_\iota+\delta z}\l(\frac{\abs\nabla}{2}\r)$ on $\rn[2]/{2\pi\zn[2]}$. Thus, the result follows from the known pseudo-differential theory.
\end{proof}

\begin{corollary}
  Let $\Gf$ and $\Hf$ be gauged polyhomogeneous operators with $\Gf(0)=\Hf(0)$.
  \begin{enumerate}
  \item[(i)] The residue $c_{-1}(\zeta(\Gf),0)$ of $\zeta(\Gf)$ in zero is gauge-invariant up to the gauge scalings $\delta_{\Gf}$ and $\delta_{\Hf}$. More precisely,
    \begin{align*}
      \frac{c_{-1}(\zeta(\Gf),0)}{\delta_{\Gf}}=\frac{c_{-1}(\zeta(\Hf),0)}{\delta_{\Hf}}.
    \end{align*}
  \item[(ii)] Let $\fa\iota\in I:\ d_\iota\ne-2$. Then, the constant Laurent coefficient is gauge-invariant, i.e.,
    \begin{align*}
      c_{0}(\zeta(\Gf),0)=c_{0}(\zeta(\Hf),0)
    \end{align*}
  \end{enumerate}
\end{corollary}

\begin{definition}
  Let $\Gf$ be a gauged polyhomogeneous operator. 
  \begin{enumerate}
  \item[(i)] $\Gf(0)$ is called non-critical if and only if $\fa\iota\in I:\ d_\iota\ne-2$.
  \item[(ii)] Let $\Gf(0)$ be non-critical. Then, we define the $\zeta$-regularized trace of $\Gf(0)$ as
    \begin{align*}
      \tr_\zeta(\Gf(0)):=\zeta(\Gf)(0).
    \end{align*}
  \end{enumerate}
\end{definition}

Since criticality, i.e., whether or not there is a $d_\iota=-2$, determines the possible existence of a pole in zero, we will use the following terminology. The lowest order Laurent coefficient of $\zeta(\Gf)$ is independent of chosen gauge, i.e., it depends only on $\Gf(0)$, justifying the definition
\begin{align*}
  \loLc(\Gf(0)):=
  \begin{cases}
    c_{-1}\l(\zeta(\Gf),0\r)&,\ \Gf(0)\text{ critical}\\
    c_{0}\l(\zeta(\Gf),0\r)&,\ \Gf(0)\text{ non-critical}
  \end{cases}
\end{align*}
independent on whether or not these values are zero.

\begin{prop}
  The lowest order Laurent coefficient is tracial given any normal gauge.

  More precisely, let $A=(\id\otimes L_A)\circ\Delta$ and $B=(\id\otimes L_B)\circ\Delta$ be polyhomogeneous operators. Then, $AB=BA$ and, if $AB$ is non-critical, $\tr_\zeta(AB)=\tr_\zeta(BA)$.
\end{prop}

\begin{proof}
  Note that $AB=BA$ since
  \begin{align*}
    \fa m,n\in\nn_0:\ ABR_{n,m}=&A\l(L_B(R_{n,m})R_{n,m}\r)
    =L_A(R_{n,m})L_B(R_{n,m})R_{n,m}\\
    =&BAR_{n,m}.
  \end{align*}
  Let $\Hf$ be a gauged polyhomogeneous operator with $\Hf(0)=B$, $\Gf_1:=A\Hf$, and $\Gf_2:=\Hf A$. Then, $\Gf_1=\Gf_2$ and, hence, $\zeta(\Gf_1)=\zeta(\Gf_2)$. Since $\loLc$ is gauge independent, we obtain $\loLc(AB)=\loLc(BA)$.
\end{proof}

\begin{example*}
  Consider the Laplacian $\Delta_{\Tp}=(\id\otimes L_{\Delta_{\Tp}})\circ\Delta$ with
  \begin{align*}
    \fa m,n\in\nn_0:\ L_{\Delta_{\Tp}}\l(R_{n,m}\r)=-(n-m)^2=-n^2+2mn-m^2.
  \end{align*}
  Thus, $\Delta_{\Tp}$ is non-critical and $\tr_\zeta(\Delta_{\Tp})$ can be written as a Kontsevich-Vishik trace of a classical pseudo-differential operator which yields
  \begin{align*}
    \tr_\zeta(\Delta_{\Tp})=\trKV\l(-\frac{(\abs{\d_1}-\abs{\d_2})^2}{4}\r)=0
  \end{align*}
  where $(\d_1,\d_2)$ is the gradient on $\rn[2]/{2\pi\zn[2]}$.
  
  More generally, let $p:\ \rn[2]\to\cn$ be a polynomial. An operator $D=(\id\otimes L)\circ\Delta$ with $L(R_{n,m})=p(m,n)$ is called a differential operator. Then, $D$ is non-critical and $\tr_\zeta(D)=0$.
\end{example*}

\section{The $\zeta$-regularized heat-trace}\label{sec:heat}
For a compact Riemannian manifold $M$ without boundary and of even dimension $\dim M\in 2\nn$, the heat-trace, that is, the trace of the heat-semigroup $T$, has a polyhomogeneous expansion in the time parameter near zero which is of the form
\begin{align*}
  \tr T(t)=\frac{\vol(M)}{(4\pi t)^{\frac{\dim M}{2}}}+\frac{\mathrm{total\ curvature}(M)}{3(4\pi)^{\frac{\dim M}{2}}t^{\frac{\dim M}{2}-1}}+\mathrm{higher\ order\ terms}.
\end{align*}
More generally, for $\dim M\in\nn$, the heat-trace has an expansion
\begin{align*}
  \tr T(t)=(4\pi t)^{-\frac{\dim M}{2}}\sum_{k\in\nn_0}t^{\frac{k}{2}}A_k
\end{align*}
for $t\searrow0$ where the $A_k$ are called heat-invariants. These heat-invariants are spectral invariants of Laplace type operators $\nabla^*\nabla+V$ generating the corresponding ``heat-semigroup'' where $\nabla$ is a connection on a vector bundle over $M$ and $V$ is called the potential.

In this section, we want to consider the heat-semigroup $T$ generated by $\Delta_{\Tp}$ on $\Tp$ and compute the heat-coefficient. However, while the heat-semigroup is a semigroup of trace class operators, this is no longer true for the Toeplitz algebra since
\begin{align*}
  \sigma(T(t))=\sigma_p(T(t))=\l\{\exp\l(-t(n-m)^2\r);\ m,n\in\nn_0\r\}
\end{align*}
including multiplicities. In other words, each eigenvalue has multiplicity $\aleph_0$, i.e., $T(t)$ is bounded but not compact. Thus, we need to regularize $\tr T(t)$. However, if we na\"ively expand $\exp\l(-t(n-m)^2\r)=\sum_{k\in\nn_0}\frac{(-t)^k(n-m)^{2k}}{k!}$, we obtain a polyhomogeneous representation which fails to satisfy $\sup_{\iota\in I}\Re(d_{\iota})<\infty$. In other words, we cannot simply apply the theory developed in section~\ref{sec:zeta}. Thus, we define a slightly more general $\zeta$-function.

\begin{definition}
  Let $\Gf=\l(\Gf(z)\r)_{z\in\cn[2]}$ be a holomorphic family of operators on $\Tp$. Then, we call $\Gf$ a generalized gauged polyhomogeneous operator (or generalized gauge) if and only if $\Gf(0)=1$ and $\fa z\in\cn[2]:\ \Gf(z)=\Gf_0(z)+(\id\otimes L_{p,\Gf(z)})\circ\Delta$ where $\Gf_0$ is of trace class for all $z\in\cn[2]_{\Re(\cdot)<R}$ for some $R\in\rn_{>0}$ and 
  \begin{align*}
    \fa z\in\cn[2]\ \fa m,n\in\nn_0:\ L_{p,\Gf(z)}(R_{n,m})=\sum_{\iota\in I}\alpha_\iota(z)\sigma_{d_{1,\iota}+\delta_1z_1}(m)\sigma_{d_{2,\iota}+\delta_2z_2}(n).
  \end{align*}
  Furthermore, let $H$ be a polyhomogeneous operator on $\Tp$. Then, we define $\zeta(H,\Gf)$ to be the maximal holomorphic extension of $z\mapsto\tr(H\Gf(z))$ with open, connected domain containing $\cn[2]_{\Re(\cdot)<r}$ for some $r\in\rn$ sufficiently small. 
\end{definition}

\begin{corollary}
  Let $H=(\id\otimes L_H)\circ\Delta$ and $\Gf=(\id\otimes L_{\Gf})\circ\Delta$ a generalized gauge. Then, 
  \begin{align*}
    \fa m,n\in\nn_0\ \fa z\in\cn[2]:\ H\Gf(z)R_{n,m}=L_H\l(R_{n,m}\r)L_{\Gf}(z)\l(R_{n,m}\r).
  \end{align*}
\end{corollary}

Eventually, we are interested in $\zeta(H,\Gf)(z,z)$ in a neighborhood of $z=0$. However, with the introduction of a second complex parameter, we can compute the limit $\lim_{z\to0}\zeta(H,\Gf)(z,z)$ by computing either $\lim_{z_1\to0}\lim_{z_2\to0}\zeta(H,\Gf)(z_1,z_2)$ or $\lim_{z_2\to0}\lim_{z_1\to0}\zeta(H,\Gf)(z_1,z_2)$ which may be significantly easier. This is possible since the identity theorem holds for holomorphic functions on $\Omega\sse\cn[n]$ in the usual sense, that is, if $\Omega$ is open and connected and a holomorphic function $f$ vanishes in an open subset of $\Omega$, then $f=0$. Thus, the $\zeta$-function in multiple variables is unique. Furthermore, since restricting a generalized gauge to the diagonal yields a gauge again, it suffices to check gauge independence of the lowest order Laurent coefficient at zero using gauges parametrized on $\cn$.

\begin{lemma}\label{lemma-gauge-independence}
  Let $A$ be an operator on $\Tp$ and $\Gf$ and $\Hf$ gauged polyhomogeneous operators with $\Gf(0)=\Hf(0)$, $L_{\Gf(z)}\l(R_{n,m}\r)=\sigma_{\delta z}(m,n)$, and $L_{\Hf(z)}\l(R_{n,m}\r)=\tilde\sigma_{\delta z}(m,n)$. Then, the lowest order Laurent coefficient of $\zeta(A,\Gf)$ and $\zeta(A,\Hf)$ in zero coincide, i.e., depend only on $A$.
\end{lemma}

\begin{proof}
  Let $l$ be the order of the lowest order Laurent coefficient of $\zeta(A,\Gf)$ and $\zeta(A,\Hf)$. Then, $\If(z):=\frac{\Gf(z)-\Hf(z)}{z}$ is also a gauged polyhomogeneous operator and the order of the lowest order Laurent coefficient of $\zeta(A,\If)$ is $l$. In particular, $z\mapsto z^l\zeta(A,\If)(z)=z^{l-1}\l(\zeta(A,\Gf)(z)-\zeta(A,\Hf)(z)\r)$ is holomorphic in zero, i.e., the lowest order Laurent coefficients of $\zeta(A,\Gf)$ and $\zeta(A,\Hf)$ must coincide.
\end{proof}

Thus, we can compute the leading order coefficient of the $\zeta$-regularized heat-trace on $\Tp$.

\begin{theorem}\label{thm:heat-trace-toeplitz}
  Let $\fa m,n\in\nn_0\ \fa z\in\cn[2]:\ L_{\Gf}(z)\l(R_{n,m}\r)=\sum_{\iota\in I}\alpha_\iota(z)n^{\delta_{1,\iota}z_1}m^{\delta_{2,\iota}z_2}$ with finite $I$, $\Gf=(\id\otimes L_{\Gf})\circ\Delta$ such that $\Gf(0)=1$, and $T$ the heat-semigroup on $\Tp$ (generated by $\Delta_{\Tp}$). Then, the following assertions are true.
  \begin{enumerate}
  \item[(i)] $\max\l\{\frac{\Re(z_1)}{\delta_{1,\iota}},\frac{\Re(z_2)}{\delta_{2,\iota}};\ \iota\in I\r\}<-1\ \then\ \fa t\in\rn_{\ge0}:\ T(t)\Gf(z_1,z_2)$ is of trace class.
  \item[(ii)] For all $t\in\rn_{>0}$, we obtain
    \begin{align*}
      \lim_{z_2\to0}\lim_{z_1\to0}\zeta(T(t),\Gf)(z_1,z_2)=-\frac{1}{2}-\sum_{k\in\nn}(k+1)e^{-tk^2}
    \end{align*}
    and
    \begin{align*}
      \Htr_{\Tp,\zeta,\Gf}(t)=\frac{1}{2}-\sum_{k\in\nn}(k-1)e^{-tk^2}.
    \end{align*}
  \item[(iii)] $\lim_{t\searrow0}\ 4\pi t\ \Htr_{\Tp,\zeta,\Gf}(t)=-2\pi$.
  \end{enumerate}
\end{theorem}

\begin{proof}
  ``(i)'' The assertion follows directly from the fact that $I$ is finite and each $\abs{L_{T(t)}\l(R_{n,m}\r)}\le1$.

  \paragraph{} ``(ii)'' Setting $k:=n-m$ we observe for $\Re(z_1)$ and $\Re(z_2)$ sufficiently small
  \begin{align*}
    \zeta(T(t),\Gf)(z)=&\sum_{m\in\nn}\sum_{n\in\nn}\sum_{\iota\in I}\alpha_\iota(z)e^{-t(n-m)^2}n^{\delta_{1,\iota}z_1}m^{\delta_{2,\iota}z_2}\\
    =&\sum_{m\in\nn}\sum_{k\in\zn_{>-m}}\sum_{\iota\in I}\alpha_\iota(z)e^{-tk^2}(m+k)^{\delta_{1,\iota}z_1}m^{\delta_{2,\iota}z_2}\\
    =&\sum_{m\in\nn}\l(\sum_{k=1-m}^0+\sum_{k\in\nn}\r)\sum_{\iota\in I}\alpha_\iota(z)e^{-tk^2}(m+k)^{\delta_{1,\iota}z_1}m^{\delta_{2,\iota}z_2}.
  \end{align*}
  Let us consider the $\sum_{k=1-m}^0$ case first. To assist readability, we will notationally suppress $\sum_{\iota\in I}\alpha_\iota(z)$ since $I$ is finite. Then,
  \begin{align*}
    \sum_{m\in\nn}\sum_{k=1-m}^0e^{-tk^2}(m+k)^{\delta_{1,\iota}z_1}m^{\delta_{2,\iota}z_2}=&\sum_{m\in\nn}\sum_{k=0}^{m-1}e^{-tk^2}(m-k)^{\delta_{1,\iota}z_1}m^{\delta_{2,\iota}z_2}\\
    =&\sum_{k\in\nn_0}e^{-tk^2}\sum_{m\in\nn_{>k}}(m-k)^{\delta_{1,\iota}z_1}m^{\delta_{2,\iota}z_2}
  \end{align*}
  is holomorphic in $z_1$ in a neighborhood of $0$ given $\Re(z_2)<-\frac{1}{\delta_{2,\iota}}$ and the limit $z_1\to0$ yields
  \begin{align*}
    \lim_{z_1\to0}\sum_{m\in\nn}\sum_{k=1-m}^0e^{-tk^2}(m+k)^{\delta_{1,\iota}z_1}m^{\delta_{2,\iota}z_2}=&\sum_{k\in\nn_0}e^{-tk^2}\sum_{m\in\nn_{>k}}m^{\delta_{2,\iota}z_2}\\
    =&\sum_{k\in\nn_0}e^{-tk^2}\l(\zeta_R(-\delta_{2,\iota}z_2)-\sum_{m=1}^km^{\delta_{2,\iota}z_2}\r)
  \end{align*}
  which itself has a holomorphic extension in $z_2$ to a neighborhood of $0$, i.e.,
  \begin{align*}
    \begin{aligned}
      \lim_{z_2\to0}\lim_{z_1\to0}\sum_{m\in\nn}\sum_{k=1-m}^0e^{-tk^2}(m+k)^{\delta_{1,\iota}z_1}m^{\delta_{2,\iota}z_2}=&\sum_{k\in\nn_0}e^{-tk^2}\l(\zeta_R(0)-k\r)\\
      =&-\sum_{k\in\nn_0}e^{-tk^2}\frac{2k+1}{2}.
    \end{aligned}
  \end{align*}
  Considering the $\sum_{k\in\nn}$ term, we still have holomorphy in $z_1$ in a neighborhood of $0$ given $\Re(z_2)<-\frac{1}{\delta_{2,\iota}}$ and, thus,
  \begin{align*}
    \lim_{z_2\to0}\lim_{z_1\to0}\sum_{k\in\nn}\sum_{m\in\nn}e^{-tk^2}(m+k)^{\delta_{1,\iota}z_1}m^{\delta_{2,\iota}z_2}=&\lim_{z_2\to0}\sum_{k\in\nn}\sum_{m\in\nn}e^{-tk^2}m^{\delta_{2,\iota}z_2}\\
    =&\lim_{z_2\to0}\sum_{k\in\nn}e^{-tk^2}\zeta_R(-\delta_{2,\iota}z_2)\\
    =&-\frac{1}{2}\sum_{k\in\nn}e^{-tk^2}.
  \end{align*}
  Hence,
  \begin{align*}
    \lim_{z_2\to0}\lim_{z_1\to0}\zeta(T(t),\Gf)(z_1,z_2)=-\frac{1}{2}-\sum_{k\in\nn}(k+1)e^{-tk^2}
  \end{align*}
  since $\sum_{\iota\in I}\alpha_\iota(0)=1$.

  Since gauging terms with $m=0$ or $n=0$ yields the constant function $0$, we need to add these terms again for the heat-trace, i.e.,
  \begin{align*}
    \Htr_{\Tp,\zeta,\Gf}(t)=-\frac{1}{2}-\sum_{k\in\nn}(k+1)e^{-tk^2}+1+2\sum_{k\in\nn}e^{-tk^2}=\frac{1}{2}-\sum_{k\in\nn}(k-1)e^{-tk^2}.
  \end{align*}

  \paragraph{} ``(iii)'' It suffices to consider the $-\sum_{k\in\nn}(k+1)e^{-tk^2}$ term since the remainder is a ``classical'' heat-trace on a $1$-dimensional compact manifold without boundary and, thus, has asymptotics proportional to $t^{-\frac12}$ for $t\searrow0$. We can estimate the series using the integral comparison test since $\rn_{>0}\ni x\mapsto(x+1)e^{-tx^2}\in\rn$ is increasing on $[0,K_t]$ and decreasing on $\rn_{\ge K_t}$ where $K_t:=\frac12\sqrt{\frac{2}{t}+1}-\frac12$.

  On $[0,K_t]$, we obtain
  \begin{align*}
    \sum_{k=1}^{\lfloor K_t\rfloor-1}(k+1)e^{-tk^2}\le\int_1^{\lfloor K_t\rfloor}(x+1)e^{-tx^2}dx\le\sum_{k=2}^{\lfloor K_t\rfloor}(k+1)e^{-tk^2}
  \end{align*}
  and
  \begin{align*}
    \kappa_t:=\int_1^{\lfloor K_t\rfloor}(x+1)e^{-tx^2}dx=&\frac{\sqrt{\pi t}\erf(\sqrt t\lfloor K_t\rfloor)-e^{-t\lfloor K_t\rfloor^2}}{2t}-\frac{\sqrt{\pi t}\erf(\sqrt t)-e^{-t}}{2t}
  \end{align*}
  where $\erf$ denotes the error function (with range $[-1,1]$). Then, we obtain
  \begin{align*}
    \sum_{k=1}^{\lfloor K_t\rfloor}(k+1)e^{-tk^2}\in\l[\kappa_t+2e^{-t},\kappa_t+(\lfloor K_t\rfloor+1)e^{-t\lfloor K_t\rfloor^2}\r].
  \end{align*}
  Similarly, on $\rn_{\ge K_t}$, we obtain
  \begin{align*}
    \sum_{k\in\nn_{\ge \lfloor K_t\rfloor+2}}(k+1)e^{-tk^2}\le\int_{\rn_{\ge \lfloor K_t\rfloor+1}}(x+1)e^{-tx^2}dx\le\sum_{k\in\nn_{\ge \lfloor K_t\rfloor+1}}(k+1)e^{-tk^2}
  \end{align*}
  and
  \begin{align*}
    \lambda_t:=\int_{\rn_{\ge \lfloor K_t\rfloor+1}}(x+1)e^{-tx^2}dx=&\frac{\sqrt{\pi}}{2\sqrt t}-\frac{\sqrt{\pi t}\erf(\sqrt t(\lfloor K_t\rfloor+1))-e^{-t(\lfloor K_t\rfloor+1)^2}}{2t}
  \end{align*}
  which yields
  \begin{align*}
    \sum_{k\in\nn_{\ge \lfloor K_t\rfloor+1}}(k+1)e^{-tk^2}\in\l[\lambda_t,\lambda_t+(\lfloor K_t\rfloor+2)e^{-t(\lfloor K_t\rfloor+1)^2}\r].
  \end{align*}
  Hence,
  \begin{align*}
    \sum_{k\in\nn}(k+1)e^{-tk^2}-\kappa_t-\lambda_t\in\l[2e^{-t},(\lfloor K_t\rfloor+1)e^{-t\lfloor K_t\rfloor^2}+(\lfloor K_t\rfloor+2)e^{-t(\lfloor K_t\rfloor+1)^2}\r].
  \end{align*}

  In order to compute $\lim_{t\searrow0}t\sum_{k\in\nn}(k+1)e^{-tk^2}$, we need to study $\lim_{t\searrow0}\sqrt t\lfloor K_t\rfloor$ first. Let $\lfloor K_t\rfloor=n$. Then,
  \begin{align*}
    n\le K_t=\frac12\sqrt{\frac{2}{t}+1}-\frac12<n+1
  \end{align*}
  implies
  \begin{align*}
    \frac{2}{(2n+3)^2-1}<t\le\frac{2}{(2n+1)^2-1}.
  \end{align*}
  In other words,
  \begin{align*}
    \lfloor K_t\rfloor=n\quad \then\quad\sqrt t\lfloor K_t\rfloor\in\l(\frac{n\sqrt2}{\sqrt{(2n+3)^2-1}},\frac{(n+1)\sqrt2}{\sqrt{(2n+1)^2-1}}\r]
  \end{align*}
  which yields
  \begin{align*}
    \lim_{t\searrow0}\sqrt t\lfloor K_t\rfloor\in\l[\lim_{n\to\infty}\frac{\sqrt2}{\sqrt{\frac{(2n+3)^2-1}{n^2}}},\lim_{n\to\infty}\frac{\sqrt2}{\sqrt{\frac{(2n+1)^2-1}{(n+1)^2}}}\r]=\{1\}.
  \end{align*}
  Finally, this implies
  \begin{align*}
    \lim_{t\searrow0}t\kappa_t=&\lim_{t\searrow0}\frac{\sqrt{\pi t}\erf(\sqrt t\lfloor K_t\rfloor)-e^{-t\lfloor K_t\rfloor^2}}{2}-\frac{\sqrt{\pi t}\erf(\sqrt t)-e^{-t}}{2}=\frac{1-e^{-1}}{2}\\
    \lim_{t\searrow0}t\lambda_t=&\lim_{t\searrow0}\frac{\sqrt{\pi t}}{2}-\frac{\sqrt{\pi t}\erf(\sqrt t(\lfloor K_t\rfloor+1))-e^{-t(\lfloor K_t\rfloor+1)^2}}{2}=\frac{e^{-1}}{2},
  \end{align*}
  i.e.,
  \begin{align*}
    \lim_{t\searrow0}4\pi t\ \Htr_{\Tp,\zeta,\Gf}(t)=\lim_{t\searrow0}-2\pi t-4\pi t\sum_{k\in\nn}(k+1)e^{-tk^2}=-4\pi\lim_{t\searrow0}t\kappa_t+t\lambda_t=-2\pi.
  \end{align*}
\end{proof}

\begin{remark*}
  This result is very interesting as well, since it says the ``heat-invariants'' and ``criticality'' in section~\ref{sec:zeta} indicate that $\Tp$ is a ``quantum manifold'' of dimension $2$ with ``volume'' $-2\pi$. This stands in stark contrast to the classical background. The Brownian motion on $\Tp$ is induced by Brownian motion on the circle $\rn/{2\pi\zn}$.  Yet, we do not observe ``criticality'' $-1$ and $\lim_{t\searrow0}\sqrt{4\pi t}\Htr_{\Tp,\zeta,\Gf}(t)\in\rn\setminus\{0\}$ which would be expected if Brownian motion in $\Tp$ inherited the properties of Brownian motion in $\rn/{2\pi\zn}$. 

  However, there is a way of making sense of these results. $\Tp_0$ can be seen as the semigroup algebra of the bicyclic semigroup and $\Tp$ as the universal inverse semigroup C*-algebra of the bicyclic semigroup. In a sense we can view $\Tp$ as the ``generalized Pontryagin type dual'' of the bicyclic semigroup which would give rise to the dimension $2$.
\end{remark*}

\section{The discrete Heisenberg group algebras}\label{sec:heisenberg}
In this section, we want to consider the $\zeta$-regularized trace and $\zeta$-heat-trace on the group algebra generated by the discrete Heisenberg groups.

\begin{definition}\label{def:heisenberg}
  Let $N\in\nn$ and $P_j$, $Q_j$ ($j\in\nn_{\le N}$), and $Z$ be unitaries satisfying $P_iZ=ZP_i$, $Q_jZ=ZQ_j$, $P_iP_j=P_jP_i$, $Q_iQ_j=Q_jQ_i$, $P_iQ_j=Q_jP_i$ for $i\ne j$, and $P_iQ_i=ZQ_iP_i$.
  
  The discrete Heisenberg group algebra $\Hn_N$ of dimension $2N+1$ is the universal C*-algebra generated by $\l\{R_{m,n,p}:=P^mQ^nZ^p;\ m,n\in\zn[N], p\in\zn\r\}$. $\Hn_N$ carries the structure of a C*-bialgebra by extending the maps $\Delta(R_{m,n,p})=R_{m,n,p}\otimes R_{m,n,p}$ and $\eps(R_{m,n,p})=1$.
\end{definition}

\begin{prop}
Let $H$ be a Hilbert space, $\eta_{P_j}, \eta_{Q_j}$ be elements of $H$ such that
$$\eta_P:=(\langle\eta_{P_i},\eta_{P_j} \rangle)_{1\leq i,j\leq N}\in M_N(\rn),\quad \eta_Q:=(\langle\eta_{Q_i},\eta_{Q_j} \rangle)_{1\leq i,j\leq N}\in M_N(\rn)$$
and $\Im(\langle\eta_{P_i},\eta_{Q_j} \rangle)$ is constant over $i,j\in \{1,\dots N\}$ and let $\lambda_{P_j}$ and $\lambda_{Q_j}$ be real numbers then there exists a unique Gaussian Sch\" urmann triple on the discrete Heisenberg group algebra such that
\begin{align*}
\eta(P_j)&=\eta_{P_j},&\eta(Q_j)&=\eta_{Q_j},\\
L(P_j-P_j)&=2i \lambda_{P_j} & L(Q_j-Q_j)&=2i \lambda_{Q_j}.
\end{align*}
Furthermore every Gaussian Sch\"urmann triple on the discrete Heisenberg group algebra arises this way.
\end{prop}

\begin{proof}
The method of proof is identical to that of Theorem \ref{thm:isolp}.

Given a Gaussian Sch\"urmann triple on $\Hn_N$ the commutativity conditions and the product rule on $L$ give the real constraint on the matrices $\eta_P$ and $\eta_Q$. Note that for Gaussian cocycles the relation $PQ=ZQP$ implies that $\eta(Z)=0$.

The constant imaginary constraint, $\Im\langle \eta_{P_i},\eta_{Q_j}\rangle$, is a result of the product rule on $L$ and the identity $L(P_{i}Q_{j})=L(ZQ_{j}P_{i})$ which implies that  $L(Z)=\langle \eta_{P_i},\eta_{Q_j}\rangle-\langle \eta_{Q_j},\eta_{P_i}\rangle$ for all $i,j$.
\end{proof}

\begin{corollary}
Every Gaussian generating functional on the discrete Heisenberg group algebra is of the form 
$$L(P^mQ^nZ^p)=i(m\cdot\lambda_P+n\cdot \lambda_Q+p \lambda_Z)-\frac{1}{2}\begin{pmatrix}m&n\end{pmatrix}\begin{pmatrix}\eta_P &\eta_{PQ}\\ \eta_{PQ}^t& \eta_Q\end{pmatrix}\begin{pmatrix}m\\n\end{pmatrix}$$
for all $m,n\in \zn[N]$ and $p\in \zn$ where $\eta_{PQ}:=(\langle\eta_{P_i},\eta_{Q_j}\rangle)_{1\leq i,j\leq N}$.
\end{corollary}

\begin{example*}
Note that the generating functional in the previous Corollary strongly resembles the exponent of the characteristic function of a multivariate normal distribution. Furthermore, if $\lambda_Z=0$ then $\eta_{PQ}\in M_N(\rn)$ and the corresponding L\'evy process can be restricted to the classical $2N$-torus via the quotient by the ideal generated by $Z-I$.

In this case the vectors $\lambda_P$ and $\lambda_Q$ dictate the drift in various directions and the matrices $\eta_P,\eta_Q$ and $\eta_{PQ}$ give us information of the covariance. The canonical choice of Brownian motion on a multidimensional object is without drift and should consist of independent one-dimensional Brownian motions in each direction. 

In this scenario, this can be achieved by the choice $H=\cn[2N]$, $\eta_{P_i}=e_i$, $\eta_{Q_j}=e_{N+j}$, $\lambda_{P_i}=\lambda_{Q_j}=0$. Thus, 
  \begin{align*}
    \fa \mu=(m,n)\in\zn[2N]\ \fa p\in\zn:\ L(R_{\mu,p})=-\frac{1}{2}\norm{\mu}_{\ell_2(2N)}^2.
  \end{align*}
\end{example*}

This warrants the following definition of polyhomogeneous operators on $\Hn_N$.
\begin{definition}
  An operator $H:=(\id\otimes L)\circ \Delta$ on $\Hn_N$ is called polyhomogeneous if and only if $\ex r\in\rn\ \ex I\sse\nn\ \ex\alpha\in \ell_1(I)\ \ex d\in(\cn[2N+1]_{\Re(\cdot)<r})^I\ \fa \mu\in\zn[2N+1]:$
  \begin{align*}
    L(R_\mu)=\sum_{\iota\in I}\alpha_\iota \mu^{d_\iota}
  \end{align*}
  where we assume that each $\sum_{\iota\in I}\alpha_\iota \mu^{d_\iota}$ is absolutely convergent.
\end{definition}

Similar to the Toeplitz case, we obtain that the spectrum of a polyhomogeneous operator $H$ is given by the closure of the point spectrum
\begin{align*}
  \sigma_p(H)=\l\{L\l(R_\mu\r);\ \mu\in\zn[2N+1]\r\}
\end{align*}
and we can define gauged polyhomogeneous operators in a similar fashion.

\begin{definition}
  Let $\Gf$ be a poly-holomorphic family of operators satisfying 
  \begin{align*}
    \fa z\in\cn:\ \Gf(z)=(\id\otimes L(z))\circ\Delta
  \end{align*}
  such that each $\l\{L(z)(R_\mu);\ \mu\in\zn[2N+1]\r\}$ is closed and each $\Gf(z)$ is polyhomogeneous with
  \begin{align*}
    \fa \mu\in\zn[2N+1]:\ L(z)(R_\mu)=\sum_{\iota\in I}\alpha_\iota(z)\sigma_{d_\iota+\delta z}(\mu)
  \end{align*}
  where each $\alpha_\iota$ is holomorphic, $\sigma_{d}:\ \rn[2N+1]\to\cn$ is homogeneous of degree $d$, i.e., $\fa\lambda\in\rn_{>0}:\ \sigma_d(\lambda\cdot)=\lambda^d\sigma_d$, and $\delta\in\rn_{>0}$. Then, we call $\Gf$ a gauged polyhomogeneous operator with index set $I$.

  Furthermore, we call $\Gf$ normally gauged if and only if $\delta=1$.
\end{definition}

We can then show that $\zeta$-functions exist.
\begin{theorem}\label{thm:existence-zeta-function-heisenberg}
  Let $\Gf$ be a gauged polyhomogeneous operator with
  \begin{align*}
    \fa \mu\in\zn[2N+1]:\ L(z)(R_\mu)=\sum_{\iota\in I}\alpha_\iota(z)\sigma_{d_\iota+\delta z}(\mu).
  \end{align*}
  Then, $\Gf(z)$ is of trace class if $\fa \iota\in I:\ \Re(d_{\iota}+\delta z)<-2N-1$ and the $\zeta$-function $\zeta(\Gf)$ defined by meromorphic extension of $z\mapsto\tr\Gf(z)$ has isolated first order poles in the set $\l\{\frac{-2N-1-d_{\iota}}{\delta};\ \iota\in I\r\}$. Furthermore, the lowest order Laurent coefficient is tracial.
\end{theorem}

\begin{proof}
  Consider the torus $T:=\rn[2N+1]/{2\pi\zn[2N+1]}$. Then, $\Gf(z)$ has the same spectrum as $D(z):=\sum_{\iota\in I}\alpha_\iota(z)(i\d)^{d_\iota+\delta_\iota z}$ where $\d_k$ is the derivative in the $k$\textsuperscript{th} coordinate direction. In particular, we obtain that $D(z)$ is of trace class if all $\Re(d_\iota+\delta z)<-2N-1$ and since $D(z)$ is a pseudo-differential operator on $T$, we obtain the assertion from the established pseudo-differential theory.
\end{proof}

\begin{corollary}
  Let $\Gf$ be a gauged differential operator on $\Hn[N]$. Then, $I$ is finite and all $d_\iota\in\nn_0$. Then, $\zeta(\Gf)=0$.
\end{corollary}

Similarly, the heat-trace in $\Hn_N$ can be reduced to the heat-trace in $\rn[2N]/{2\pi\zn[2N]}$.

\begin{theorem}  
  Let $T$ be the heat-semigroup on $\Hn_N$, $S$ the heat-semigroup on $\rn[2N]/{2\pi\zn[2N]}$, and $\Gf$ a gauged polyhomogeneous operator on $\Hn_N$ with $\Gf(0)=1$. Then, the $\zeta$-regularized heat-trace $\Htr_{\Hn_N,\zeta,\Gf}$ on $\Hn_N$ satisfies
  \begin{align*}
    \Htr_{\Hn_N,\zeta,\Gf}(t)=\zeta(T(t),\Gf)(0)=-\tr S(t).
  \end{align*}
  In particular, the $k$\textsuperscript{th} heat coefficient $A_k(\Hn_N)$ of $\Hn_N$ is $-A_k(\rn[2N]/{2\pi\zn[2N]})$ where $A_k(\rn[2N]/{2\pi\zn[2N]})$ is the $k$\textsuperscript{th} heat coefficient of $\rn[2N]/{2\pi\zn[2N]}$.
\end{theorem}

\begin{proof}
  Using the same argument as in Lemma~\ref{lemma-gauge-independence}, we obtain gauge independence of $\zeta(T(t),\Gf)(0)$ and can choose $\delta',\delta''\in\rn_{>0}$ such that $\fa\mu\in\zn[2N]\ \fa p\in\zn:\ L_{\Gf(z)}(R_{\mu,p})=\norm{\mu}_{\ell_2(2N)}^{\delta'z}\abs p^{\delta'' z}$. Hence, there exists a gauged polyhomogeneous operator $\Hf$ on $\rn[2N]/{2\pi\zn[2N]}$ with $\Hf(0)=1$ such that
  \begin{align*}
    \tr T(t)\Gf(z)=&\sum_{p\in\zn}\ubr{\sum_{\mu\in\zn[2N]}e^{-t\norm\mu_{\ell_2(2N)}^2}\norm{\mu}_{\ell_2(2N)}^{\delta'z}}_{\text{gauged trace of }S(t)}\abs p^{\delta'' z}\\
    =&2\tr(S(t)\Hf(z))\sum_{p\in\nn}\abs p^{\delta'' z}\\
    =&2\zeta_R(-\delta''z)\tr S(t)\Hf(z)
  \end{align*}
  where $\zeta_R$ is the Riemann $\zeta$-function, i.e., $\zeta_R(0)=-\frac12$.
\end{proof}

\section{The discrete Heisenberg group algebras with $Z\in\cn$}\label{sec:heisenberg-Z-complex}

In section~\ref{sec:heisenberg} we considered $\Hn_N$ as generated by the $P_i$, $Q_j$, and $Z$. However, if $Z\in\cn$, then $Z$ is not a generator and we obtain the ``reduced'' algebra $\Hn[r]_N$. Thus, $\Hn[r]_N$ has the generators $\{P^mQ^n;\ m,n\in\zn\}$ and we obtain the following two theorems.

\begin{theorem}\label{thm:existence-zeta-function-heisenberg-reduced}
  Let $\Gf$ be a gauged polyhomogeneous operator with
  \begin{align*}
    \fa \mu\in\zn[2N]:\ L(z)((P,Q)^\mu)=\sum_{\iota\in I}\alpha_\iota(z)\sigma_{d_\iota+\delta z}(\mu).
  \end{align*}
  Then, $\Gf(z)$ is of trace class if $\fa \iota\in I:\ \Re(d_{\iota}+\delta z)<-2N$ and the $\zeta$-function $\zeta(\Gf)$ defined by meromorphic extension of $z\mapsto\tr\Gf(z)$ has isolated first order poles in the set $\l\{\frac{-2N-d_{\iota}}{\delta};\ \iota\in I\r\}$. Furthermore, the lowest order Laurent coefficient is tracial.
\end{theorem}

\begin{theorem}\label{thm:reduced-heisenberg-heat-trace}
  Let $T$ be the heat-semigroup on $\Hn[r]_N$ and $S$ the heat-semigroup on $\rn[2N]/{2\pi\zn[2N]}$. Then, $\fa t\in\rn_{>0}:\ T(t)$ is of trace class and
  \begin{align*}
    \Htr_{\Hn[r]_N}(t)=\tr T(t)=\tr S(t).
  \end{align*}
  In particular, the $k$\textsuperscript{th} heat coefficient $A_k(\Hn[r]_N)$ of $\Hn[r]_N$ coincides with the $k$\textsuperscript{th} heat coefficient $A_k(\rn[2N]/{2\pi\zn[2N]})$ of $\rn[2N]/{2\pi\zn[2N]}$.
\end{theorem}

\begin{remark*}
  These results do not come as a surprise, since $\Hn[r]_1$ is the non-commutative $2$-torus $A_\theta$ generated by two unitaries $U$ and $V$ satisfying $UV=e^{-2\pi i\theta}VU$ where $\theta\in\rn$. As such the family $(A_\theta)_{\theta\in\rn}$ is a fundamental class of examples of non-commutative spaces generalizing the algebra of continuous functions on the $2$-torus; a property recovered by the Brownian motion approach to the heat-trace.
\end{remark*}

\section{The non-commutative torus}\label{sec:non-com-torus}

Having observed the non-commutative $2$-torus as a special case of the discrete Heisenberg group, we want to continue studying more general non-commutative tori. This will also give us a direct means of comparison with the classical heat-trace approach on the non-commutative torus $\Tn[n]_\theta$ as studied in \cite{azzali-levy-neira-paycha,levy-neira-paycha}. There, too, the heat-trace recovered the dimension of the underlying torus, i.e., the commutative case $\Tn[n]_0=\rn[n]/{2\pi\zn[n]}$.

Let us first recall the usual construction of the non-commutative torus. Given a real symmetric $N\times N$ matrix $\theta$ and unitaries $U_k$ with $k\in\zn[N]$, $U_0=1$, and
\begin{align*}
  \fa m,n\in\zn[N]:\ U_mU_n=e^{-\pi i\langle m,\theta n\rangle_{\ell_2(N)}}U_{m+n},
\end{align*}
we consider the algebra $A_\theta:=\l\{\sum_{k\in\zn[N]}a_k U_k;\ a\in\Sp(\zn[N])\r\}$ where $\Sp$ denotes the Schwartz space. Then, it is possible to define a corresponding algebra of pseudo-differential operators which has been extensively studied in~\cite{levy-neira-paycha}. The for us interesting operator is the Laplace $\Delta_\theta:=\sum_{j=1}^N\d_j^2$ where $\d_j\sum_{k\in\zn[N]}a_k U_k:=\sum_{k\in\zn[N]}k_ja_k U_k$.

Since we need our algebra to be a C*-bialgebra, we consider the C*-algebra $\Ap[N]_\theta$ generated by the $U_k$ ($k\in\zn[N]$) as well as a finite set of additional (unitary) generators $\l\{Z_\tau;\ \tau\in {\Tf}\r\}$ (${\Tf}\in\nn$), and set $\eps(U_k)=\eps(Z_\tau)=1$, $\Delta(U_k)=U_k\otimes U_k$, and $\Delta(Z_\tau)=Z_\tau\otimes Z_\tau$. The $Z_\tau$ are a generalized version of the $e^{-\pi i\theta_{k,l}}$ which we consider to be generators as well (for now). Thus, the Laplacian $\Delta_{\theta}$ has the generating functional $L$ which satisfies
\begin{align*}
  \fa p\in \zn[{\Tf}]\ \fa k\in \zn[N]:\ L_{\Delta_{\theta}}(Z^pU_k)=-\norm k_{\ell_2(N)}^2.
\end{align*}
This follows from $\Delta_\theta\sum_{k\in\zn[N]}a_k U_k=\sum_{k\in\zn[N]}\sum_{j=1}^N k_j^2 a_k U_k$ and is consistent with the Brownian motion approach (cf. example below Definition~\ref{def:heisenberg}). In particular, we obtain similar results to the Heisenberg group algebra case.
\begin{theorem}\label{thm:existence-zeta-function-nc-torus}
  Let $\Gf$ be a gauged polyhomogeneous operator on $\Ap[N]_\theta$ with
  \begin{align*}
    \fa p\in \zn[{\Tf}]\ \fa k\in \zn[N]:\ L(z)(Z^pU_k)=\sum_{\iota\in I}\alpha_\iota(z)\sigma_{d_\iota+\delta z}((p,k)).
  \end{align*}
  Then, $\Gf(z)$ is of trace class if $\fa \iota\in I:\ \Re(d_{\iota}+\delta z)<-N-{\Tf}$ and the $\zeta$-function $\zeta(\Gf)$ defined by meromorphic extension of $z\mapsto\tr\Gf(z)$ has isolated first order poles in the set $\l\{\frac{-N-{\Tf}-d_{\iota}}{\delta};\ \iota\in I\r\}$. Furthermore, the lowest order Laurent coefficient is tracial.
\end{theorem}

\begin{corollary}
  Let $\Gf$ be a gauged differential operator on $\Ap[N]_\theta$. Then, $I$ is finite and all $d_\iota\in\nn_0$. Then, $\zeta(\Gf)=0$.
\end{corollary}

\begin{theorem}
  Let $T$ be the heat-semigroup on $\Ap[N]_\theta$, $S$ the heat-semigroup on $\rn[N]/{2\pi\zn[N]}$, and $\Gf$ a gauged polyhomogeneous operator on $\Ap[N]_\theta$ with $\Gf(0)=1$. Then, the $\zeta$-regularized heat-trace $\Htr_{\Ap[N]_\theta,\zeta,\Gf}$ on $\Ap[N]_\theta$ satisfies
  \begin{align*}
    \Htr_{\Ap[N]_\theta,\zeta,\Gf}(t)=\zeta(T(t),\Gf)(0)=(-1)^{\Tf}\tr S(t).
  \end{align*}
  In particular, the $k$\textsuperscript{th} heat coefficient $A_k(\Ap[N]_\theta)$ of $\Ap[N]_\theta$ is $(-1)^{\Tf}A_k(\rn[N]/{2\pi\zn[N]})$ where $A_k(\rn[N]/{2\pi\zn[N]})$ is the $k$\textsuperscript{th} heat coefficient of $\rn[N]/{2\pi\zn[N]}$.
\end{theorem}

Similarly, for $A^N_\theta$, i.e., the case $Z\in\cn[{\Tf}]$, we obtain the following analogous theorems.
\begin{theorem}
  Let $\Gf$ be a gauged polyhomogeneous operator on $A^N_\theta$ with
  \begin{align*}
    \fa k\in \zn[N]:\ L(z)(U_k)=\sum_{\iota\in I}\alpha_\iota(z)\sigma_{d_\iota+\delta z}(k).
  \end{align*}
  Then, $\Gf(z)$ is of trace class if $\fa \iota\in I:\ \Re(d_{\iota}+\delta z)<-N$ and the $\zeta$-function $\zeta(\Gf)$ defined by meromorphic extension of $z\mapsto\tr\Gf(z)$ has isolated first order poles in the set $\l\{\frac{-N-d_{\iota}}{\delta};\ \iota\in I\r\}$. Furthermore, the lowest order Laurent coefficient is tracial.
\end{theorem}

\begin{corollary}
  Let $\Gf$ be a gauged differential operator on $A^N_\theta$. Then, $I$ is finite and all $d_\iota\in\nn_0$. Then, $\zeta(\Gf)=0$.
\end{corollary}

\begin{theorem}
  Let $T$ be the heat-semigroup on $A^N_\theta$ and $S$ the heat-semigroup on $\rn[N]/{2\pi\zn[N]}$. Then, $\fa t\in\rn_{>0}:\ T(t)$ is of trace class and
  \begin{align*}
    \Htr_{A^N_\theta}(t)=\tr T(t)=\tr S(t).
  \end{align*}
  In particular, the $k$\textsuperscript{th} heat coefficient $A_k(A^N_\theta)$ of $A^N_\theta$ coincides with the $k$\textsuperscript{th} heat coefficient $A_k(\rn[N]/{2\pi\zn[N]})$ of $\rn[N]/{2\pi\zn[N]}$.
\end{theorem}

\section{Gaussian invariants of $SU_q(2)$}\label{sec:SUq2}

Finally, we want to have a look at the quantum group $SU_q(2)$. This case is particularly interesting since there is a canonical choice of Brownian motion for the classical case $SU_1(2)=SU(2)$ but not necessarily for $SU_q(2)$ for $q\ne 1$. There is a unique driftless Gaussian (up to time scaling) on each $SU_q(2)$ for $q\ne 1$ which we will treat as the heat-semigroup even though it can not be the heat-semigroup on $SU(2)$. This will allow us to compute a $\zeta$-regularized trace and recover that $SU_q(2)$ is $3$-dimensional.

We will begin with a quick summary of Section 6.2 of \cite{timmermann}. In order to construct $SU_q(2)$, let us start with the compact Lie group 
\begin{align*}
  SU(2):=\l\{g_{\alpha,\gamma}:=
  \begin{pmatrix}
    \alpha&-\gamma^*\\
    \gamma&\alpha^*
  \end{pmatrix}\in B\l(\cn[2]\r);\ \alpha,\gamma\in\cn,\ \det g_{\alpha,\gamma}=1\r\}
\end{align*}
and define $a,c\in C(SU(2))$ by $a(g_{\alpha,\gamma}):=\alpha$ and $c(g_{\alpha,\gamma})=\gamma$. Then, the C*-algebra generated by $a$ and $c$ subject to $a^*a+c^*c=1$ is a C*-algebraic compact quantum group with co-multiplication $\Delta$ given by
\begin{align*}
  \Delta(a)=a\otimes a+c\otimes c\qquad\text{and}\qquad\Delta(c)=c\otimes a+a^*\otimes c,
\end{align*}
co-unit $\eps(a)=1$, $\eps(c)=0$, and antipode $S(a)=a^*$, $S(a^*)=a$, $S(c)=-c$, and $S(c^*)=-c^*$.

\begin{definition}
  Let $q\in[-1,1]\setminus\{0\}$. Then, we define $SU_q(2)$ to be the universal unital C*-algebra generated by elements $a$ and $c$ subject to the condition that
  \begin{align*}
    u:=
    \begin{pmatrix}
      a&-qc^*\\c&a^*
    \end{pmatrix}
  \end{align*}
  is unitary, i.e., 
  \begin{enumerate}

  \item[(i)] $a^*a+c^*c=1$
  \item[(ii)] $aa^*+q^2c^*c=1$
  \item[(iii)] $c^*c=cc^*$
  \item[(iv)] $ac=qca$
  \item[(v)] $ac^*=qc^*a$
  \end{enumerate}
  Furthermore, $SU_q(2)$ is endowed with the co-unit $\eps$ given by $\eps(a)=1$ and $\eps(c)=0$, and co-multiplication given by
  \begin{align*}
    \Delta(a)=a\otimes a-q c^*\otimes c\qquad\text{and}\qquad\Delta(c)=c\otimes a+a^*\otimes c.
  \end{align*}
\end{definition}

If we let $SU^0_q(2)\sse SU_q(2)$ denote the *-subalgebra generated by $a$ and $c$ and $\Delta_0$ and $\eps_0$ the co-multiplication $\Delta$ and co-unit $\eps$ restricted to this *-subalgebra then $(SU^0_q(2), \Delta_0,\eps_0 ) = (SU_q(2), \Delta,\eps)_0$ is the associated algebraic compact quantum group to $(SU_q(2), \Delta,\eps)$. 

Furthermore, the family $(a_{kmn})_{(k,m,n)\in\zn\times\nn_0\times\nn_0}$ defined as
\begin{align*}
  a_{k,m,n}:=
  \begin{cases}
    a^k(c^*)^m c^n&,\ k\in\nn_0\\
    (a^*)^{-k}(c^*)^m c^n&,\ k\in-\nn
  \end{cases}
\end{align*}
is a basis of $SU_q^0(2)$. 

The Gaussian generating functionals on $SU_q(2)$ for $|q|\in (0,1)$ are classified in~\cite{schurmann-skeide}. For the quantum group we have a family of characters $\eps_\phi:\ SU_q(2)\to \cn$ for such that
\begin{align*}
  \eps_\phi(a)=e^{i\phi}\quad \text{ and }\quad \eps_{\phi}(c)=0.
\end{align*}
This family of characters is pointwise continuous with respect to $\phi$ and satisfies $\eps_0=\eps$. On the basis $\{a_{kmn};\ (k,m,n)\in\zn\times\nn_0\times\nn_0\}$ of the dense subalgebra we have that $\eps_\phi(a_{kmn})=e^{ik\phi}\delta_{m+n,0}$ for $\phi\in \rn$. 

We will define linear functionals $\eps'(a_{kmn})=\d_\phi\eps_\phi(a_{kmn})|_{\phi=0}=ik\delta_{m+n,0}$ and $\eps''(a_{kmn})=\d_\phi^2\eps_\phi(a_{kmn})|_{\phi=0}=-k^2\delta_{m+n,0}$.

\begin{prop}
  All Gaussian generating functionals are of the form
  \begin{align*}
    L=r_D\eps'+r\eps''
  \end{align*}
  where $r_D\in \rn$ and $r\in \rn_{>0}$.
\end{prop}
By the definition of drift, the parameter $r_D$ contributes only to drift so we will only consider $r_D=0$. This leaves only positive multiples of $\eps''$.

\begin{prop}
  The operator $T_L:=(\id\otimes \eps'')\circ\Delta:\ SU_q(2)\to SU_q(2)$ takes the following values on $\{a_{kmn};\ (k,m,n)\in\zn\times\nn_0\times\nn_0\}$
  \begin{align*}
    T_L(a_{kmn})=-(k-m+n)^2a_{kmn}.
  \end{align*}
\end{prop}

\begin{proof}

  First note that 
  \begin{align*}
    \Delta(a_{kmn})=(a\otimes a-qc^*\otimes c)^k(c^*\otimes a^*+a\otimes c^*)^m(c\otimes a+a^*\otimes c)^n
  \end{align*}
  for all $(k,m,n)\in \zn\times \nn_0\times \nn_0$ and, using the commutation rules, we can simplify so that
  \begin{align*}
    a^ka^{*m}a^n=a^{k-m+n}+\text{ terms with } c.
  \end{align*}
  Then, by applying $\eps''$ to the right leg of the tensor product, we observe that the only non-zero term is given by $a^{k-m+n}$, i.e.,
  \begin{align*}
     \ T_L(a_{kmn})=a^kc^{*m}c^n\epsilon''(a^ka^{*m}a^n)=-(k-m+n)^2a_{kmn}.
  \end{align*}
  for all $(k,m,n)\in\zn\times \nn_0\times \nn_0$.
\end{proof}

\begin{corollary}
  Let $\Gf$ be a gauged polyhomogeneous operator with
  \begin{align*}
    \fa (k,m,n)\in\zn\times\nn_0\times\nn_0:\ \Gf(z)(a_{kmn})=\sum_{\iota\in I}\alpha_\iota(z)\sigma_{d_\iota+\delta z}((k,m,n))a_{kmn}.
  \end{align*}
  Then, $\Gf(z)$ is of trace class if $\fa\iota\in I:\ \Re(d_\iota+\delta z)<-3$ and the $\zeta$-function $\zeta(\Gf)$ defined by meromorphic extension of $z\mapsto\tr\Gf(z)$ has at most isolated first order poles in the set $\l\{\frac{-3-d_\iota}{\delta};\ \iota\in I\r\}$. Furthermore, the lowest order Laurent coefficient is tracial.
\end{corollary}
\begin{proof}
  This, again, follows directly from the fact that the spectrum of $\Gf(z)$ and $\sum_{\iota\in I}\alpha_\iota(z)\sigma_{d_\iota+\delta z}\l(\l(i\d_1,\frac{\abs{\d_2}}{2},\frac{\abs{\d_3}}{2}\r)\r)$ on $\rn[3]/{2\pi\zn[3]}$ coincide.
\end{proof}

Now taking the operator exponentiation we see that the operator semigroup $T(t):\ SU_q(2)\to SU_q(2)$ associated with $r\eps''$ is given by 
\begin{align*}
  T(t)(a_{kmn})=e^{-rt(k-m+n)^2}a_{kmn}
\end{align*}
on the basis $\{a_{kmn};\ (k,m,n)\in\zn\times\nn_0\times\nn_0\}$. It is important to note that this is not the heat-semigroup on $SU(2)$. Since $SU(2)$ is a compact Riemannian $C^\infty$-manifold without boundary, its heat-semigroup is a semigroup of trace class operators but each $T(t)$ above has multiplicity $\aleph_0$ for each of its eigenvalues. In other words, none of the $T(t)$ is compact.

\begin{theorem}\label{thm:gaussian-trace-SUq(2)}
  Let $T$ be a driftless Gaussian semigroup on $SU_q(2)$, i.e.,
  \begin{align*}
    \fa (k,m,n)\in\zn\times\nn[2]_0\ \fa t\in\rn_{>0}:\ T(t)(a_{kmn})=e^{-rt(k-m+n)^2}a_{kmn},
  \end{align*}
  and $\l(\Gf_t(z)\r)_{z\in\cn[3]}$ a holomorphic family of operators on $SU_q(2)$ satisfying
  \begin{align*}
    \Gf_t(z)a_{kmn}=e^{-rt(k+m-n)^2}\abs{k}^{\delta_1z_1}m^{\delta_2z_2}n^{\delta_3z_3}
  \end{align*}
  for all $z\in\cn[3]$, $(k,m,n)\in\zn\times\nn_0\times\nn_0$, and $t\in\rn_{>0}$, where $\delta_1,\delta_2,\delta_3\in\rn_{>0}$. Then,
  \begin{align*}
    \tr_\zeta(T(t))=&\zeta(\Gf_t)(0)=\frac{1}{12}+\frac{13}{12}\sum_{k\in\nn}e^{-rtk^2}+\frac{13}{12}\sum_{k\in\nn}ke^{-rtk^2}-\sum_{k\in\nn}k^2e^{-rtk^2}
  \end{align*}
  and
  \begin{align*}
    \lim_{t\searrow0}(4\pi t)^{\frac{3}{2}}\tr_\zeta(T(t))=-2\pi^2r^{-\frac{3}{2}}.
  \end{align*}
\end{theorem}

\begin{proof}
  In order to show 
  \begin{align*}
    \tr_\zeta(T(t))=\zeta(\Gf_t)(0)=\frac{1}{12}+\frac{13}{12}\sum_{k\in\nn}e^{-rtk^2}+\frac{13}{12}\sum_{k\in\nn}ke^{-rtk^2}-\sum_{k\in\nn}k^2e^{-rtk^2},
  \end{align*}
  we need to compute the limit $z\to0$ of
  \begin{align*}
    \sum_{k\in\zn}\sum_{m\in\nn}\sum_{n\in\nn}e^{-rt(k-m+n)^2}\abs{k}^{\delta_1z_1}m^{\delta_2z_2}n^{\delta_3z_3}
  \end{align*}
  which we can alternatively write as
  \begin{align*}
    \sum_{k\in\zn}e^{-rtk^2}\abs{k}^{\delta_1z_1}\sum_{m\in\nn}\sum_{n\in\nn}e^{-rt(n-m)^2}e^{-2rtk(n-m)}m^{\delta_2z_2}n^{\delta_3z_3}.
  \end{align*}
  The inner two series are very similar to the heat-trace on the Toeplitz algebra - there is simply an additional factor $e^{-2rtk(n-m)}$ now. Hence, we will treat these series in a similar fashion.
  \begin{align*}
    \sum_{m\in\nn}\sum_{n\in\nn}e^{-rt(n-m)^2}e^{-2rtk(n-m)}m^{\delta_2z_2}n^{\delta_3z_3}=\sum_{m\in\nn}\sum_{\ell\in\zn_{>-m}}e^{-rt\ell^2}e^{-2rtk\ell}m^{\delta_2z_2}(m+\ell)^{\delta_3z_3}
  \end{align*}
  which allows us to change gauge with respect to $z_3$ to obtain for $\Re(z_2)\ll0$ and we obtain
  \begin{align*}
    \sum_{m\in\nn}\sum_{\ell\in\zn_{>-m}}e^{-rt\ell^2}e^{-2rtk\ell}m^{\delta_2z_2}\abs{\ell}^{\delta_3z_3}.
  \end{align*}
  Let 
  \begin{align*}
    A:=\sum_{m\in\nn}\sum_{\ell=1-m}^0e^{-rt\ell^2}e^{-2rtk\ell}m^{\delta_2z_2}\abs{\ell}^{\delta_3z_3}
  \end{align*}
  and 
  \begin{align*}
    B:=\sum_{m\in\nn}\sum_{\ell\in\nn}e^{-rt\ell^2}e^{-2rtk\ell}m^{\delta_2z_2}\abs{\ell}^{\delta_3z_3}.
  \end{align*}
  Then, we obtain
  \begin{align*}
    \lim_{z_2\to0}A=&\lim_{z_2\to0}\sum_{m\in\nn}\sum_{\ell=0}^{m-1}e^{-rt\ell^2}e^{2rtk\ell}m^{\delta_2z_2}\ell^{\delta_3z_3}\\
    =&\lim_{z_2\to0}\sum_{\ell\in\nn_0}\sum_{m\in\nn_{>\ell}}e^{-rt\ell^2}e^{2rtk\ell}m^{\delta_2z_2}\ell^{\delta_3z_3}\\
    =&\lim_{z_2\to0}\sum_{\ell\in\nn_0}e^{-rt\ell^2}e^{2rtk\ell}\l(\zeta_R(-\delta_2z_2)-\sum_{m=1}^\ell m^{\delta_2z_2}\r)\ell^{\delta_3z_3}\\
    =&-\frac{1}{2}\sum_{\ell\in\nn_0}e^{-rt\ell^2}e^{2rtk\ell}\l(2\ell+1\r)\ell^{\delta_3z_3}
  \end{align*}
  and
  \begin{align*}
    \lim_{z_2\to0}B=&\sum_{m\in\nn}\sum_{\ell\in\nn}e^{-rt\ell^2}e^{-2rtk\ell}m^{\delta_2z_2}\ell^{\delta_3z_3}=-\frac{1}{2}\sum_{\ell\in\nn}e^{-rt\ell^2}e^{-2rtk\ell}\ell^{\delta_3z_3}.
  \end{align*}
  Hence, we are looking to compute
  \begin{align*}
    \lim_{z_3\to0}\lim_{z_1\to0}\sum_{k\in\zn}e^{-rtk^2}\abs{k}^{\delta_1z_1}\l(-\frac{1}{2}-\frac{1}{2}\sum_{\ell\in\nn}e^{-rt\ell^2}\ell^{\delta_3z_3}\l((2\ell+1)e^{2rtk\ell}+e^{-2rtk\ell}\r)\r).
  \end{align*}
  Using
  \begin{align*}
    \sum_{k\in\zn}e^{-rtk^2\pm2rtk\ell}=1+\sum_{k\in\nn}e^{-rtk^2+2rtk\ell}+\sum_{k\in\nn}e^{-rtk^2-2rtk\ell},
  \end{align*}
  we are looking for the limit $z_1,z_3\to0$ of
  \begin{align*}
    &-\sum_{k\in\nn}e^{-rtk^2}k^{\delta_1z_1}\\
    &-\frac{1}{2}\sum_{k,\ell\in\nn}e^{-rt(k-\ell)^2}(2\ell+1)k^{\delta_1z_1}\ell^{\delta_3z_3}-\frac{1}{2}\sum_{k,\ell\in\nn}e^{-rt(k+\ell)^2}(2\ell+1)k^{\delta_1z_1}\ell^{\delta_3z_3}\\
    &-\frac{1}{2}\sum_{k,\ell\in\nn}e^{-rt(k-\ell)^2}k^{\delta_1z_1}\ell^{\delta_3z_3}-\frac{1}{2}\sum_{k,\ell\in\nn}e^{-rt(k+\ell)^2}k^{\delta_1z_1}\ell^{\delta_3z_3}
  \end{align*}
  which, in parts, we already know in terms of $\Htr_{\Tp,\zeta}(t)=-\frac{1}{2}-\sum_{k\in\nn}(k+1)e^{-tk^2}$, the heat-trace on the Toeplitz algebra. Thus,
  \begin{align*}
    \zeta(\Gf_t)(0)=&-\sum_{k\in\nn}e^{-rtk^2}-\frac{1}{2}\sum_{k,\ell\in\nn}e^{-rt(k+\ell)^2}(2\ell+1)-\frac{1}{2}\Htr_{\Tp,\zeta}(rt)-\frac{1}{2}\sum_{k,\ell\in\nn}e^{-rt(k+\ell)^2}\\
    &-\frac{1}{2}\lim_{z_1,z_3\to0}\sum_{k,\ell\in\nn}e^{-rt(k-\ell)^2}(2\ell+1)k^{\delta_1z_1}\ell^{\delta_3z_3}\\
    =&-\sum_{k\in\nn}e^{-rtk^2}-\frac{1}{2}\sum_{k,\ell\in\nn}e^{-rt(k+\ell)^2}(2\ell+1)-\Htr_{\Tp,\zeta}(rt)-\frac{1}{2}\sum_{k,\ell\in\nn}e^{-rt(k+\ell)^2}\\
    &-\lim_{z_1,z_3\to0}\sum_{k,\ell\in\nn}e^{-rt(k-\ell)^2}k^{\delta_1z_1}\ell^{1+\delta_3z_3}    \\
    =&-\sum_{k\in\nn}e^{-rtk^2}-\frac{1}{2}\sum_{k,\ell\in\nn}e^{-rt(k+\ell)^2}(2\ell+1)-\Htr_{\Tp,\zeta}(rt)-\frac{1}{2}\sum_{k,\ell\in\nn}e^{-rt(k+\ell)^2}\\
    &-\zeta_R(-1)\sum_{k\in\nn}e^{-rtk^2}-\zeta_R(-1)+\zeta_R(0)+\l(\frac{1}{2}-\zeta_R(-1)\r)\sum_{k\in\nn}ke^{-rtk^2}\\
    &-\frac{1}{2}\sum_{k\in\nn}k^2e^{-rtk^2}\\
    =&\frac{1}{12}+\frac{1}{12}\sum_{k\in\nn}e^{-rtk^2}+\frac{19}{12}\sum_{k\in\nn}ke^{-rtk^2}-\frac{1}{2}\sum_{k\in\nn}k^2e^{-rtk^2}\\
    &-\sum_{k,\ell\in\nn}e^{-rt(k+\ell)^2}\ell-\sum_{k,\ell\in\nn}e^{-rt(k+\ell)^2}
  \end{align*}
  where we computed the final limit in the same way the limit in the proof of Theorem~\ref{thm:heat-trace-toeplitz}. Finally, the latter two series can be reduced to the former three; namely,
  \begin{align*}
    \sum_{k,\ell\in\nn}e^{-rt(k+\ell)^2}=&\sum_{\ell\in\nn}\sum_{m\in\nn_{>\ell}}e^{-rtm^2}=\sum_{m\in\nn_{\ge2}}\sum_{\ell=1}^{m-1}e^{-rtm^2}=\sum_{m\in\nn}me^{-rtm^2}-\sum_{m\in\nn}e^{-rtm^2}
  \end{align*}
  and
  \begin{align*}
    \sum_{k,\ell\in\nn}e^{-rt(k+\ell)^2}\ell=\sum_{m\in\nn_{\ge2}}\sum_{\ell=1}^{m-1}e^{-rtm^2}\ell=\frac{1}{2}\sum_{m\in\nn}m^2e^{-rtm^2}-\frac{1}{2}\sum_{m\in\nn}me^{-rtm^2}.
  \end{align*}
  Hence, we obtain
  \begin{align*}
    \zeta(\Gf_t)(0)=&\frac{1}{12}+\frac{13}{12}\sum_{k\in\nn}e^{-rtk^2}+\frac{13}{12}\sum_{k\in\nn}ke^{-rtk^2}-\sum_{k\in\nn}k^2e^{-rtk^2}.
  \end{align*}

  To obtain the asymptotics with respect to $t\searrow0$, we shall approximate each series using the integral comparison test again which yields
  \begin{align*}
    \lim_{t\searrow0}\sqrt{4\pi t}\sum_{k\in\nn}e^{-rtk^2}=\frac{\pi}{\sqrt{r}},
  \end{align*}
  \begin{align*}
    \lim_{t\searrow0}4\pi t\sum_{k\in\nn}ke^{-rtk^2}=\frac{2\pi}{r},
  \end{align*}
  and
  \begin{align*}
    \lim_{t\searrow0}(4\pi t)^{\frac{3}{2}}\sum_{k\in\nn}k^2e^{-rtk^2}=\frac{2\pi^2}{r^{\frac{3}{2}}}.
  \end{align*}
  In other words,
  \begin{align*}
    \lim_{t\searrow0}(4\pi t)^{\frac{3}{2}}\tr_\zeta(T(t))=-\frac{2\pi^2}{r^{\frac{3}{2}}}.
  \end{align*}

\end{proof}

\begin{remark*}
  Recall that we have not been computing ``heat-invariants'' in Theorem~\ref{thm:gaussian-trace-SUq(2)} since there is no Brownian motion on $SU_q(2)$. Thus, we cannot interpret $SU_q(2)$ as a ``quantum manifold'' of volume $-2\pi^2r^{-\frac{3}{2}}$. However, the driftless Gaussians that we still have at our disposal recovered the pole order $\frac32$ for $t\searrow0$ which is the expected result since $SU(2)$ is isomorphic to the (real) $3$-sphere. In other words, we can interpret $SU_q(2)$ as a three-dimensional ``quantum manifold'' which gives the correct limit at $q=1$. This stands in contrast to Connes' observation~\cite{connes-hochschild-dimension} that the Hochschild dimension of $SU_q(2)$ drops from $3$ ($q=1$) to $1$ ($q\ne1$). However, it is consistent with Hadfield's and Krähmer's results~\cite{hadfield-kraehmer-I,hadfield-kraehmer-II} that $SU_q(2)$ is a twisted $3$-dimensional Calabi-Yau algebra.
\end{remark*}

\section{Conclusion}\label{sec:conclusion}
We have considered driftless Gaussians, in particular Brownian motion, on a number of C*-bialgebras to define Laplace-type operators and heat-semigroups. We spectrally regularized their traces using operator $\zeta$-functions and computed quantities like criticality and heat-coefficients. We noticed that the notion of dimension obtained from the critical degree of homogeneity need not coincide with the dimension obtained from the heat-trace. In particular, having an abstract twist structure is seen in the dimension obtained from criticality but not in the heat-trace. Thus, the ``criticality dimension'' seems (somewhat unsurprisingly) to be related to the algebraic properties of the algebra whereas the ``heat-trace dimension'' (being induced by the dynamics of Brownian motion) seems to be related to geometric/analytic properties of the algebra. This is particularly obvious in the case of twisted classical structures where the ``heat-trace dimension'' coincides with the classical dimension which is not the case for the ``criticality dimension'' which also counts the number of abstract twists (as these are generators of the algebra as well). In the $SU_q(2)$ case, we observed the additional obstruction that there is no Brownian motion and we had to make do with the projectively unique generator of a driftless Gaussian as our version of a ``Laplacian'' (whose $SU(2)$ version does not have compact resolvent). Still, we were able to recover $3$-dimensionality using the ``Gauss-trace'' and criticality.

In terms of the ``heat-coefficients'', the leading order coefficient can hardly be interpreted as a volume since many of them are negative. On the other hand, we observed that the ``heat-coefficients'' can be used to differentiate between different algebras (e.g., the Toeplitz algebra and the discrete Heisenberg group algebra $\Hn_1$ have different ``heat-coefficients''). However, it is not possible to ``hear the shape of a quantum drum'' using these ``heat-coefficients'' alone as we have observed that the ``heat-coefficients'' of $\Hn_N$ and $A_\theta^N$ with complex twists coincide with the heat-coefficients of classical tori.


\begin{bibdiv}
  \begin{biblist}
  \bib{accardi-schurmann-waldenfels}{article}{
    AUTHOR = {ACCARDI, L.},
    AUTHOR ={SCH\"URMANN, M.},
    AUTHOR={VON WALDENFELS, W.},
     TITLE = {Quantum independent increment processes on superalgebras},
   JOURNAL = {Mathematische Zeitschrift},
    VOLUME = {198},
      YEAR = {1988},
    NUMBER = {4},
     PAGES = {451--477},
      ISSN = {0025-5874},
       DOI = {10.1007/BF01162868},
       URL = {http://dx.doi.org/10.1007/BF01162868},
     }
  
    \bib{atiyah-bott-patodi}{article}{
      author={ATIYAH, M.},
      author={BOTT, R.},
      author={PATODI, V. K.},
      title={On the heat equation and the index theorem},
      journal={Inventiones Mathematicae},
      volume={19 (4)},
      pages={279-330},
      date={1973}
    }
    \bib{aukhadiev-grigoryan-lipacheva}{article}{
      author={AUKHADIEV, M. A.},
      author={GRIGORYAN, S. A.},
      author={LIPACHEVA, E. V.},
      title={Infinite-Dimensional Compact Quantum Semigroup},
      journal={Lobachevskii Journal of Mathematics},
      volume={32},
      pages={304-316},
      date={2011}
    }
    \bib{azzali-levy-neira-paycha}{article}{
      author={AZZALI, S.},
      author={L{\'E}VY, C.},
      author={NEIRA-JIMEN{\'E}Z, C.},
      author={PAYCHA, S.},
      title={Traces of holomorphic families of operators on the noncommutative torus and on Hilbert modules},
      journal={Geometric Methods in Physics: XXXIII Workshop 2014},
      pages={3-38},
      date={2015}
    }
    \bib{belton-wills}{article}{
    AUTHOR = {BELTON, A. C. R.},
    AUTHOR={WILLS, S. J.},
     TITLE = {An algebraic construction of quantum flows with unbounded
              generators},
   JOURNAL = {Ann. Inst. Henri Poincar\'e Probab. Stat.},
    VOLUME = {51},
      YEAR = {2015},
    NUMBER = {1},
     PAGES = {349--375},
	}
    \bib{carey-gayral-rennie-sukochev}{article}{
      author={CAREY, A. L.},
      author={GAYRAL, V.},
      author={RENNIE, A.},
      author={SUKOCHEV, F. A.},
      title={Integration on locally compact noncommutative spaces},
      journal={Journal of Functional Analysis},
      volume={263},
      number={2},
      pages={383--414},
      date={2012}
    }
    \bib{carey-rennie-sadaev-sukochev}{article}{
      author={CAREY, A. L.},
      author={RENNIE, A.},
      author={SADAEV, A.},
      author={SUKOCHEV, F. A.},
      title={The Dixmier trace and asymptotics of zeta functions},
      journal={Journal of Functional Analysis},
      volume={249},
      number={2},
      pages={253--283},
      date={2007}
    }
    \bib{connes-hochschild-dimension}{article}{
      author={CONNES, A.},
      title={Cyclic cohomology, quantum group symmetries and the local index formula for $SU_q(2)$},
      journal={Journal of the Institute of Mathematics of Jussieu},
      volume={3},
      pages={17--68},
      date={2004}
    }
    \bib{connes-action-functional}{article}{
      author={CONNES, A.},
      title={The Action Functional in Non-Commutative Geometry},
      journal={Communications in Mathematical Physics},
      volume={117},
      pages={673--683},
      date={1988}
    }
    \bib{connes-fathizadeh}{article}{
      author={CONNES, A.},
      author={FATHIZADEH, F.},
      title={The term $a_4$ in the heat kernel expansion of noncommutative tori},
      journal={arXiv:1611.09815v1 [math.QA]},
      date={2016}
    }
    \bib{connes-moscovici}{article}{ 
      author={CONNES, A.},
      author={MOSCOVICI, H.},
      title={Modular curvature for noncommutative two-tori},
      journal={Journal of the American Mathematical Society},
      volume={27},
      pages={639-684},
      date={2014}
    }
    \bib{connes-tretkoff}{article}{
      author={CONNES, A.},
      author={TRETKOFF, P.},
      title={The Gauss–Bonnet theorem for the noncommutative two torus},
      journal={Noncommutative Geometry, Arithmetic, and Related Topics},
      pages={141-158},
      date={2011}
    }
    \bib{cipriani-franz-kula}{article}{
    AUTHOR = {CIPRIANI, F.},
    AUTHOR = {FRANZ, U.},
    AUTHOR = {KULA, A.},
     TITLE = {Symmetries of {L}\'evy processes on compact quantum groups,
              their {M}arkov semigroups and potential theory},
  JOURNAL = {Journal of Functional Analysis},
    VOLUME = {266},
      YEAR = {2014},
    NUMBER = {5},
     PAGES = {2789-2844},
      ISSN = {0022-1236},
       DOI = {10.1016/j.jfa.2013.11.026},
       URL = {http://dx.doi.org/10.1016/j.jfa.2013.11.026},
}
    \bib{craig}{article}{
      author={CRAIG, J. W.},
      title={A new, simple and exact result for calculating the probability of error for two-dimensional signal constellations},
      journal={Proc. 1991 IEEE Military Commun. Conf.},
      volume={2},
      pages={571-575},
      date={1991}
    }
    \bib{dabrowski-sitarz-asymmetric}{article}{ 
      author={DABROWSKI, L.},
      author={SITARZ, A.},
      title={An asymmetric noncommutative torus},
      journal={SIGMA},
      volume={11},
      pages={075},
      date={2015}
    }
    \bib{dabrowski-sitarz-curved}{article}{ 
      author={DABROWSKI, L.},
      author={SITARZ, A.},
      title={Curved noncommutative torus and Gauss–Bonnet},
      journal={Journal of Mathematical Physics},
      volume={54},
      pages={013518},
      date={2013}
    }
    \bib{fathi}{book}{
      author={FATHI, A.},
      title={On certain spectral invariants of Dirac operators on noncommutative tori and curvature of the determinant line bundle for the noncommutative two torus},
      publisher={Ph.D. thesis, University of Western Ontario},
      address={London, ON},
      date={2015}
    }
    \bib{fathi-ghorbanpour-khalkhali}{article}{
      author={FATHI, A.},
      author={GHORBANPOUR, A.},
      author={KHALKHALI, M.},
      title={Curvature of the determinant line bundle for the noncommutative two torus},
      journal={arXiv:1410.0475~[math.QA]},
      pages={1--18},
      date={2014}
    }
    \bib{fathi-khalkhali}{article}{
      author={FATHI, A.},
      author={KHALKHALI, M.},
      title={On certain spectral invariants of Dirac operators on noncommutative tori},
      journal={arXiv:1504.01174v1~[math.QA]},
      pages={1--30},
      date={2015}
    }
    \bib{fathizadeh}{article}{ 
      author={FATHIZADEH, F.}, 
      title={On the scalar curvature for the noncommutative four torus},
      journal={Journal of Mathematical Physics},
      volume={56},
      pages={062303},
      date={2015}
    }
    \bib{fathizadeh-khalkhali-scalar-4}{article}{ 
      author={FATHIZADEH, F.}, 
      author={KHALKHALI, M.},
      title={Scalar curvature for noncommutative four-tori},
      journal={Journal of Noncommutative Geometry},
      volume={9},
      pages={473-503},
      date={2015}
    }
    \bib{fathizadeh-khalkhali-scalar-2}{article}{ 
      author={FATHIZADEH, F.}, 
      author={KHALKHALI, M.},
      title={Scalar curvature for the noncommutative two torus},
      journal={Journal of Noncommutative Geometry},
      volume={7},
      pages={1145-1183},
      date={2013}
    }
    \bib{fathizadeh-khalkhali-gauss-bonnet}{article}{ 
      author={FATHIZADEH, F.}, 
      author={KHALKHALI, M.},
      title={The Gauss-Bonnet theorem for noncommutative two tori with a general conformal structure},
      journal={Journal of Noncommutative Geometry},
      volume={6},
      pages={457-480},
      date={2012}
    }
    \bib{franz}{book}{,
      AUTHOR = {FRANZ, U.},
      TITLE = {L\'evy processes on quantum groups and dual groups},
      BOOKTITLE = {Quantum independent increment processes. {II}},
      SERIES = {Lecture Notes in Mathematics},
      VOLUME = {1866},
      PAGES = {161-257},
      PUBLISHER = {Springer, Berlin},
      YEAR = {2006},
      DOI = {10.1007/11376637_3},
      URL = {http://dx.doi.org/10.1007/11376637_3},
    }
    \bib{gayral-iochum-vassilevich}{article}{
      author={GAYRAL, V.},
      author={IOCHUM, B.},
      author={VASSILEVICH, D. V.},
      title={Heat Kernel and Number Theory on NC-torus},
      journal={Communications in Mathematical Physics},
      volume={273},
      pages={415-443},
      date={2007}
    }
    \bib{guillemin}{article}{
      author={GUILLEMIN, V.},
      title={Gauged Lagrangian Distributions},
      journal={Advances in Mathematics},
      volume={102},
      pages={184-201},
      date={1993}
    }
    \bib{hadfield-kraehmer-I}{article}{
      author={HADFIELD, T.},
      author={KR{\"{A}}HMER, U.},
      title={Twisted homology of quantum $SL(2)$},
      journal={K-Theory},
      volume={34},
      pages={327-360},
      date={2005}
    }
    \bib{hadfield-kraehmer-II}{article}{
      author={HADFIELD, T.},
      author={KR{\"{A}}HMER, U.},
      title={Twisted homology of quantum $SL(2)$ - Part II},
      journal={Journal of K-Theory},
      volume={6},
      pages={69-98},
      date={2010}
    }
    \bib{hartung-phd}{book}{
      author={HARTUNG, T.},
      title={$\zeta$-functions of Fourier Integral Operators},
      publisher={Ph.D. thesis, King's College London},
      address={London},
      date={2015}
    }
    \bib{hartung-scott}{article}{
      author={HARTUNG, T.},
      author={SCOTT, S.},
      title={A generalized Kontsevich-Vishik trace for Fourier Integral Operators and the Laurent expansion of $\zeta$-functions},
      journal={arXiv:1510.07324v2~[math.AP]},
      date={2015}
    }
    \bib{hawking}{article}{
      author={HAWKING, S. W.},
      title={Zeta Function Regularization of Path Integrals in Curved Spacetime},
      journal={Communications in Mathematical Physics},
      volume={55},
      pages={133--148},
      date={1977}
    }
    \bib{iochum-masson}{article}{
      author={IOCHUM, B.},
      author={MASSON, T.},
      title={Heat asymptotics for nonminimal Laplace type operators and application to noncommutative tori},
      journal={arXiv:1707.09657 [math.DG]},
      date={2017}
    }
    \bib{kaad-senior}{article}{
      author={KAAD, J.},
      author={SENIOR, R.},
      title={A twisted spectral triple for quantum $SU(2)$},
      journal={Journal of Geometry and Physics},
      volume={62},
      number={4},
      pages={731--739},
      date={2012}
    }
    \bib{kontsevich-vishik}{article}{
      author={KONTSEVICH, M.},
      author={VISHIK, S.},
      title={Determinants of elliptic pseudo-differential operators},
      journal={Max Planck Preprint, arXiv:hep-th/9404046},
      date={1994}
    }
    \bib{kontsevich-vishik-geometry}{article}{
      author={KONTSEVICH, M.},
      author={VISHIK, S.},
      title={Geometry of determinants of elliptic operators},
      journal={Functional Analysis on the Eve of the XXI century, Vol. I, Progress in Mathematics},
      volume={131},
      pages={173-197},
      date={1994}
    }
    \bib{levy-neira-paycha}{article}{
      author={L{\'E}VY, C.},
      author={NEIRA-JIMEN{\'E}Z, C.},
      author={PAYCHA, S.},
      title={The canonical trace and the noncommutative residue on the noncommutative torus},
      journal={Transactions of the American Mathematical Society},
      volume={368},
      pages={1051-1095},
      date={2016}
    }
    
    \bib{lindsay-skalski}{article}{
    AUTHOR = {LINDSAY, J. M.},
    AUTHOR = {SKALSKI, A. G.},
     TITLE = {Quantum stochastic convolution cocycles. {I}},
   JOURNAL = {Ann. Inst. H. Poincar\'e Probab. Statist.},
    VOLUME = {41},
      YEAR = {2005},
    NUMBER = {3},
     PAGES = {581--604},
}
    
     \bib{lindsay-skalski1}{article}{,
    AUTHOR = {LINDSAY, J. M.},
    AUTHOR = {SKALSKI, A. G.},
     TITLE = {Quantum stochastic convolution cocycles. {II}},
   JOURNAL = {Communications in Mathematical Physics},
    VOLUME = {280},
      YEAR = {2008},
    NUMBER = {3},
     PAGES = {575-610},
      ISSN = {0010-3616},
       DOI = {10.1007/s00220-008-0465-x},
       URL = {http://dx.doi.org/10.1007/s00220-008-0465-x},
     }    		
    \bib{lindsay-skalski2}{article}{,
    AUTHOR = {LINDSAY, J. M.},
    AUTHOR = {SKALSKI, A. G.},
     TITLE = {Convolution semigroups of states},
   JOURNAL = {Mathematische Zeitschrift},
    VOLUME = {267},
      YEAR = {2011},
    NUMBER = {1-2},
     PAGES = {325-339},
      ISSN = {0025-5874},
       DOI = {10.1007/s00209-009-0621-9},
       URL = {http://dx.doi.org/10.1007/s00209-009-0621-9},
    }
		
    \bib{lindsay-skalski3}{article}{,
    AUTHOR = {LINDSAY, J. M.},
    AUTHOR = {SKALSKI, A. G.},
     TITLE = {Quantum stochastic convolution cocycles {III}},
   JOURNAL = {Mathematische Annalen},
    VOLUME = {352},
      YEAR = {2012},
    NUMBER = {4},
     PAGES = {779-804},
      ISSN = {0025-5831},
       DOI = {10.1007/s00208-011-0656-1},
       URL = {http://dx.doi.org/10.1007/s00208-011-0656-1},
    }
    \bib{liu}{article}{
      author={LIU, Y.},
      title={Modular curvature for toric noncommutative manifolds},
      journal={arXiv:1510.04668v2 [math.OA]},
      date={2015}
    }
    \bib{maniccia-schrohe-seiler}{article}{
      author={MANICCIA, L.},
      author={SCHROHE, E.},
      author={SEILER, J.},
      title={Uniqueness of the Kontsevich-Vishik trace},
      journal={Proceedings of the American Mathematical Society},
      volume={136 (2)},
      pages={747-752},
      date={2008}
    }
    \bib{matassa}{article}{
      author={MATASSA, M.},
      title={Non-Commutative Integration, Zeta Functions and the Haar State for $SU_q(2)$},
      journal={Mathematical Physics, Analysis and Geometry},
      volume={18},
      number={6},
      pages={1--23},
      date={2015}
    }
    \bib{murphy}{book}{
    AUTHOR = {MURPHY, G. J.},
     TITLE = {{$C^*$}-algebras and operator theory},
 PUBLISHER = {Academic Press, Inc., Boston, MA},
      YEAR = {1990},
     PAGES = {x+286},
}
    
    \bib{okikiolu}{article}{
      author={OKIKIOLU, K.},
      title={The multiplicative anomaly for determinants of elliptic operators},
      journal={Duke Mathematical Journal},
      volume={79},
      pages={722-749},
      date={1995}
    }
    \bib{paycha-scott}{article}{
      author={PAYCHA, S.},
      author={SCOTT, S. G.},
      title={A Laurent expansion for regularized integrals of holomorphic symbols},
      journal={Geometric and Functional Analysis},
      volume={17 (2)},
      pages={491-536},
      date={2007}
    }
    \bib{ray}{article}{
      author={RAY, D. B.},
      title={Reidemeister torsion and the Laplacian on lens spaces},
      journal={Advances in Mathematics},
      volume={4},
      pages={109-126},
      date={1970}
    }
    \bib{ray-singer}{article}{
      author={RAY, D. B.},
      author={SINGER, I. M.},
      title={$R$-torsion and the Laplacian on Riemannian manifolds},
      journal={Advances in Mathematics},
      volume={7},
      pages={145-210},
      date={1971}
    }
    \bib{sadeghi}{article}{
      author={SADEGHI, S.},
      title={On logarithmic Sobolev inequality and a scalar curvature formula for noncommutative tori},
      journal={Ph.D. thesis, Western University, Ontario},
      date={2016}
    }
    \bib{schurmann}{book}{
      AUTHOR = {SCH\"URMANN, M.},
      TITLE = {White noise on bialgebras},
      SERIES = {Lecture Notes in Mathematics},
      VOLUME = {1544},
      PUBLISHER = {Springer-Verlag, Berlin},
      YEAR = {1993},
      PAGES = {vi+146},
      ISBN = {3-540-56627-9},
      DOI = {10.1007/BFb0089237},
      URL = {http://dx.doi.org/10.1007/BFb0089237}
    }
    \bib{schurmann-skeide}{article}{
      author={SCH\"URMANN, M.},
      author={SKEIDE, M.},
      TITLE = {Infinitesimal generators on the quantum group {${\rm SU}_q(2)$}},
      JOURNAL = {Infin. Dimens. Anal. Quantum Probab. Relat. Top.},
      VOLUME = {1},
      YEAR = {1998},
      NUMBER = {4},
      PAGES = {573--598},
      ISSN = {0219-0257},
      URL = {http://dx.doi.org/10.1142/S0219025798000314},
    }
    \bib{scott}{article}{
      author={SCOTT, S.},
      title={The residue determinant},
      journal={Communications in Partial Differential Equations},
      volume={30},
      pages={483-507},
      number={4-6},
      date={2005}
    }
    \bib{seeley}{article}{
      author={SEELEY, R. T.},
      title={Complex Powers of an Elliptic Operator},
      journal={Proceedings of Symposia in Pure Mathematics, American Mathematical Society},
      volume={10},
      pages={288-307},
      date={1967}
    }
    \bib{sitarz}{article}{
      author={SITARZ, A.},
      title={Wodzicki residue and minimal operators on a noncommutative 4-dimensional torus},
      journal={Journal of Pseudo-Differential Operators and Applications},
      volume={5},
      pages={305-317},
      date={2014}
    }
    \bib{timmermann}{book}{
    AUTHOR = {TIMMERMANN, T.},
     TITLE = {An invitation to quantum groups and duality},
    SERIES = {EMS Textbooks in Mathematics},
      NOTE = {From Hopf algebras to multiplicative unitaries and beyond},
 PUBLISHER = {European Mathematical Society (EMS), Z\"urich},
      YEAR = {2008},
     PAGES = {xx+407},
      ISBN = {978-3-03719-043-2},
       DOI = {10.4171/043},
       URL = {http://dx.doi.org/10.4171/043}
}
    \bib{vassilevich-I}{article}{
      author={VASSILEVICH, D. V.},
      title={Heat kernel, effective action and anomalies in noncommutative theories},
      journal={Journal of High Energy Physics},
      volume={0508},
      pages={085},
      date={2005}
    }
    \bib{vassilevich-II}{article}{
      author={VASSILEVICH, D. V.},
      title={Non-commutative heat kernel},
      journal={Letters in Mathematical Physics},
      volume={67},
      pages={185-195},
      date={2004}
    }
    \bib{wodzicki}{book}{
      author={WODZICKI, M.},
      title={Noncommutative Residue. I. Fundamentals.},
      series={Lecture Notes in Mathematics},
      volume={1289},
      publisher={Springer-Verlag, Berlin},
      date={1987}
    }
  \end{biblist}
\end{bibdiv}
\end{document}